\documentclass[11pt]{article}
\usepackage[pagewise]{lineno}%\linenumbers
\usepackage{latexsym,amsfonts,amssymb,amsmath,amsthm}
\usepackage{graphicx}
\usepackage{color}
\usepackage{tikz}
\usepackage{subfigure}

\usepackage{tikz}
\usepackage{abstract}% 两栏文档，一栏摘要及关键字宏包\usepackage{bibentry}

\usepackage{url}% 超链接
\usepackage{bm}% 加粗部分公式
\usepackage{multirow}
\usepackage{booktabs}
\usepackage{epstopdf}
\usepackage{epsfig}
\usepackage{longtable}% 长表格
\usepackage{supertabular}% 跨页表格
\usepackage{changepage}% 换页
\usepackage{enumerate}% 短编号
\usepackage{caption}% 设置标题
\usepackage{indentfirst}% 中文首行缩进
\usepackage[left=2.50cm,right=2.50cm,top=2.50cm,bottom=2.50cm]{geometry}% 页边距设置
\usepackage{fancyhdr} %设置全文页眉、页脚的格式
\usepackage{algorithm}
\usepackage{algorithmicx}
\allowdisplaybreaks[4]
\usepackage{algpseudocode}

\usepackage{float}% http://ctan.org/pkg/float
\usepackage{cases}

\allowdisplaybreaks[4]
\usepackage[round,authoryear]{natbib} % 启用作者-年份格式
\usepackage[colorlinks=true,linkcolor=blue,citecolor=blue,urlcolor=blue]{hyperref}
\usepackage{setspace} 
\usepackage{geometry} %xin英文模板中输出中文
\allowdisplaybreaks
\allowdisplaybreaks
\newtheorem{theorem}{Theorem}[section]
\newtheorem{lemma}{Lemma}[section]

\newtheorem{definition}{Definition}[section]
\newtheorem{example}{Example}[section]

\newtheorem{assumption}{Assumption}[section]
\newtheorem{remark}{Remark}[section]
\numberwithin{equation}{section}

\begin{document}% 以下为正文内容

	%\maketitle% 产生标题，没有它无法显示标题
	%%%%%%%%%%%%%%%%%%%%%%%%%%%%%%%%%%%%%%%%%%%%%%%%%%%%%%%%
	\lhead{}% 页眉左边设为空
	\chead{}% 页眉中间设为空
	\rhead{}% 页眉右边设为空
	\lfoot{}% 页脚左边设为空
	\cfoot{\thepage}% 页脚中间显示页码
	\rfoot{}% 页脚右边设为空
	%%%%%%%%%%%%%%%%%%%%%%%%%%%%%%%%%%%%%%%%%%%%%%%%%%%%%%%%
	\newcommand\del[1]{}
	\title{Distributed Online Economic Dispatch with Time-Varying Coupled Inequality Constraints}
	
	\author{Yingjie Zhou\thanks{School of Mathematical Sciences, East China Normal University, Shanghai 200241, China. Email: Zhouyingjie7@163.com.}\,\,\,  Xiaoqian Wang\thanks{Academy of Mathematics and Systems Science, Chinese Academy of Sciences, Beijing 100190, China. Email: xiaoqian.wang@amss.ac.cn. (Corresponding author)}  \,\,\, Tao Li\thanks{Key Laboratory of Management, Decision and Information Systems, Institute of Systems Science, Academy of Mathematics and Systems Science, Chinese Academy of Sciences, Beijing 100190, China; School of Mathematical Sciences, University of Chinese Academy of Sciences, Beijing 100149, China. Email: litao@amss.ac.cn.} }

    \onehalfspacing %%1.5倍行间距
    
	\maketitle

	\textbf{Abstract:}
	We investigate the distributed online economic dispatch problem for power systems with time-varying coupled inequality constraints. The problem is formulated as a distributed online optimization problem in a multi-agent system. At each time step, each agent only observes its own instantaneous objective function and local inequality constraints; agents make decisions online and cooperate to minimize the sum of the time-varying objectives while satisfying the global coupled constraints. To solve the problem, we propose an algorithm based on the primal-dual approach combined with constraint-tracking. Under appropriate assumptions that the objective and constraint functions are convex, their gradients are uniformly bounded, and the path length of the optimal solution sequence grows sublinearly, we analyze theoretical properties of the proposed algorithm and prove that both the dynamic regret and the constraint violation are sublinear with time horizon $T$. Finally, we evaluate the proposed algorithm on a time-varying economic dispatch problem in power systems using both synthetic data and Australian Energy Market data. The results demonstrate that the proposed algorithm performs effectively in terms of tracking performance, constraint satisfaction, and adaptation to time-varying disturbances, thereby providing a practical and theoretically well-supported solution for real‑time distributed economic dispatch.

	
	\vskip 2mm
	\section{Introduction}
	
	With the large-scale integration of distributed energy resources, renewable generation, and demand response, the operational characteristics of modern distribution grids and microgrids have become increasingly time-varying and uncertain. Generation outputs, loads, and market prices in power systems fluctuate rapidly, making economic dispatch methods based on static or offline models unable to meet the accuracy and robustness requirements of real‑time operation. At the same time, centralized dispatch requires collecting private information from all units and relies on high-bandwidth communication infrastructure, which faces practical limitations such as data privacy concerns, a single point of failure, and communication. Therefore, in scenarios where information is locally available and an online response to time-varying disturbances is required, distributed online economic dispatch methods provide a more practical and resilient solution.

    Distributed online economic dispatch problems are increasingly prevalent in real-world operations research (OR) applications — from distribution networks and microgrids to electric-vehicle charging, demand-response programs, and virtual power plants. These settings require scalable, privacy-preserving, and resilient decision-making that can react to fast-changing generation, load, and price signals without relying on a central controller. Distributed online economic dispatch commonly operates downstream of the predict‑then‑optimize (PTO) pipelines \citep{Elmachtoub2022,Devos2025,Petropoulos2022,Petropoulos2025,wang2025}. The distributed online method takes forecasts as inputs to the per‑time minimization and leverages predicted values to make optimal scheduling decisions for upcoming time steps. Distributed online dispatch is of substantial practical importance, as it enhances operational reliability and economic efficiency, mitigates constraint violations and the need for costly corrective actions, preserves data privacy among heterogeneous agents, and supports scalable participation in energy markets and ancillary service programs. These features enable more effective real-time decision-making in modern power systems and related OR applications.

	In the distributed online economic dispatch problem, we are interested in how each generating unit determines its output at each time step based solely on locally observable information, with the goal of minimizing operating costs while ensuring the supply–demand balance. As a motivating example, we consider a microgrid composed of $N$ generators, where the decision variable for generator $i$ at time $t$ is its real‑time power output $x_{i,t}$. At each time, the system cooperatively seeks to solve the following online minimization problem:
	\begin{equation}\label{online_dispatch}
		\begin{aligned}
			\min _{x_{i,t}\in \Omega_i } & \sum_{t=1}^{T}\sum_{i=1}^N a_{i,t} x_{i,t}^2 + b_{i,t}x_{i,t} + c_{i,t} - P_t x_{i,t},\\
			\text{s.t.}\,\, &  D_t - \sum_{i=1}^N x_{i,t} \leq 0 ,\,\, t=1, \dots, T,
		\end{aligned}
	\end{equation}
	where $\Omega_i \subset \mathbb{R}$ denotes the power output range of generator $i$, $P_t$ is the real-time electricity price, $a_{i,t},b_{i,t},c_{i,t}$ are the generation cost coefficients, and $D_t$ is the system load demand at time $t$. The term $P_t x_{i,t}$ in the objective function represents the revenue or market settlement at the real‑time price, which is subtracted from the generation cost to compute the net operating cost. The coupled inequality constraint $D_t - \sum_{i=1}^N x_{i,t} \leq 0$ ensures that the total generation reaches or exceeds the system load at each time step. Since the generation cost in (\ref{online_dispatch}) is increasing and strongly convex in $x_{i,t}$, any surplus generation increases the total cost. Hence, the optimal solution satisfies the coupled constraint at equality, i.e., $D_t - \sum_{i=1}^N x_{i,t} = 0$.
    Note that, at the decision-making stage, both $P_t$ and $D_t$ are time-varying and unobserved, and thus are replaced by their forecasts. Moreover, each unit can typically observe only its own cost and constraint information, without direct knowledge of the global demand or the private parameters of other units, making the problem highly challenging. Existing methods for problem (\ref{online_dispatch}) are largely heuristic and lack rigorous theoretical analysis \citep{Guo2007,cortes2018genetic,Sunwen2022}; furthermore, some studies consider illustrative examples that do not arise from practical engineering scenarios, limiting their practical applicability \citep{xiesiyu2022,Lijueyou2022}.

    \subsection{Related works}
	In fact, (\ref{online_dispatch}) can be formulated as a distributed online optimization problem. In recent years, driven by rapid advances in communication and computation technologies, together with the widespread deployment of low‑cost devices, distributed optimization has attracted substantial attention across a broad range of research areas. Many practical application scenarios, such as decision making, statistical learning, sensor networks, resource allocation, and formation control \citep{luo2023optimal,Abouarghoub2018,ChenT2017,MaW2018,Bedi2018}, can be naturally formulated as distributed optimization tasks over multi-agent networks. Unlike traditional centralized optimization methods, distributed optimization is performed on a network of agents, each of which has access only to its private information, while no central coordinator possesses complete system knowledge. Consequently, no single agent can solve the global optimization problem on its own; instead, the global objective can be achieved only through information exchange and coordinated decision-making among agents.
	
	In practical applications, agents often operate in dynamic environments \citep{ho2018primal,ho2019exploiting}, which has motivated extensive research on distributed online optimization. For unconstrained problems, a substantial body of work has been developed \citep{Akbari2015,Mateos2014}. When constraints are present, various online distributed algorithms have been proposed, including online distributed saddle‑point algorithms \citep{Kopple2015}, online distributed mirror descent \citep{Shahrampour2018}, online distributed dual-averaging \citep{Hosseini2016}, and zero-order online distributed projection-free algorithms \citep{Lu2023}, etc. 
	To avoid the computational and storage costs associated with projection-based methods, recent work has turned to distributed online convex optimization with long-term constraints \citep{Yuan2018,Yuan2022}, where constraints are not coupled but are required to be satisfied in a long-term sense.
	For strongly convex objectives, \cite{Yuan2018} proposed a primal–dual online algorithm, while \cite{Yuan2022} studied the general convex case.  \cite{Lee2017} proposed a distributed saddle‑point approach for problems involving multiple coupled constraints and derived regret bounds, but their results heavily rely on the strong assumption of bounded dual variables. To relax this restriction, \cite{Lixiuxian2020} improved the method of \cite{Lee2017} by removing the boundedness assumption on the dual variables. Nevertheless, the method presents somewhat worse performance in terms of regret and constraint violation. For time-varying but uncoupled constraints, \cite{Yi2025} derived improved theoretical bounds on static regret and constraint violation for distributed online convex optimization, and in particular achieved tighter network cumulative constraint violation bounds under Slater's condition for the first time.
	Additional studies on time-varying inequality constraints include \cite{Yixinlei20202,Du2024}.

    Compared with other types of constraints, coupled constraints, which involve multiple decision variables or agents, are particularly common in many practical applications, such as plug-in electric vehicle (PEV) charging scheduling \citep{Vujanic2016,LiJ2018} and wireless network routing under channel uncertainty \citep{Lee2017}. However, distributed online convex optimization problems with coupled constraints have not yet been studied systematically and in depth in the existing literature.
	For distributed online optimization with time-varying coupled inequality constraints, \cite{Lijueyou2022} proposed a distributed primal-dual online algorithm with both gradient-feedback and gradient-free variants, and established sublinear bounds on static regret and constraint violation. However, their method requires stepsize parameters that depend on prior knowledge of the gradient bound. \cite{Luk2023} developed a distributed online algorithm for problems with coupled inequality constraints under the assumption of strong pseudo‑convexity, proving sublinear growth of dynamic regret and constraint violation under certain conditions; yet, this method requires both Slater's condition and bounded Lagrange multipliers due to the strong pseudo‑convexity assumption. More recently, \cite{wangcong2024} examined distributed online time-varying coupled‑constraint convex optimization with feedback delays, and proposed a delay-compensated distributed primal-dual push-sum algorithm, proving that the expected regret and constraint violation grow sublinearly.
	
	For the distributed online economic dispatch problem with time-varying coupled inequality constraints, the main challenges are as follows. First, the strong coupling of the inequality constraints. At each time step, every agent can observe only its own local objective function, local constraint functions, and local feasible set, making it difficult to satisfy the global constraints in a distributed manner. Second, existing theoretical analyses often rely on overly restrictive assumptions. As in the literature mentioned above, regret analyses often rely on strong assumptions such as bounded Lagrange multipliers and Slater's condition. Third, some algorithms require prior knowledge (for example, a uniform upper bound on the gradients of the objective functions over all time), which is difficult to obtain before the algorithm starts.

    \subsection{Our contributions}

    In this paper, we consider a general distributed online optimization problem with time-varying inequality constraints, where at each time step each agent has access only to its local objective and constraint functions, and the agents jointly minimize the sum of all local objective functions subject to coupled inequality constraints. Problem \eqref{online_dispatch} can be considered as a special case of this general setting. The main contributions of this paper are as follows. 
    \begin{itemize}
    \item  We propose a distributed online algorithm for economic dispatch based on the primal–dual framework. The proposed algorithm decouples the coupled inequality constraints by assigning each node a local copy of the Lagrange multipliers and enforcing agreement via consensus on these multipliers. Since each node requires information about the global constraint function, we extend the gradient tracking mechanism to the constraint functions so that each node can track the sum of all constraints using historical information. To avoid assuming bounded Lagrange multipliers, we introduce a regularization term into the original Lagrangian, whose coefficient effectively controls an upper bound on the multipliers.
    \item Under standard boundedness assumptions on objective and constraint functions and their gradients, we prove that both the dynamic regret and the constraint violation of the proposed algorithm grow sublinearly. Specifically, let $0<\kappa_1<\min\{2\kappa_2, 1-2\kappa_2\}$, with $\kappa_2\in (0,\frac{1}{2})$, and let $P_T$ denote the path length of the optimal solution sequence. Then, the dynamic regret admits the upper bound $\mathcal{O}\left(T^{\max\{1-\kappa_1, 1-2\kappa_2+\kappa_1, 2\kappa_2+\kappa_1\}} +T^{\max\{\kappa_1, \kappa_2\}}P_T \right)$ while the constraint violation is bounded by $\mathcal{O}\left(T^{\max\left\{1-\frac{\kappa_2}{2}, \frac{1}{2}+\frac{\kappa_1}{2}\right\}} +T^{\max\{\frac{1}{2}, \frac{1}{2}-\frac{\kappa_2}{2}+ \frac{\kappa_1}{2}\}} \sqrt{P_T} \right)$.
    \item The proposed method relaxes a common assumption in the existing literature by eliminating the need to presuppose bounded Lagrange multipliers through the use of a regularized Lagrangian \citep{Luk2023,Ayslender2000}. Moreover, it requires no prior knowledge of parameters of the global problem, thereby enhancing its practical applicability. The proposed algorithm updates decisions for the next time step using only information from the current and immediately preceding time steps, without storing all the past information, which results in low memory and computational overhead. Finally, numerical experiments  using both synthetic data and Australian Energy Market data demonstrate the practical effectiveness of the proposed approach and corroborate the theoretical performance guarantees.
    \end{itemize}

    \subsection{Brief discussion}
    In addition to the specific formulation in (\ref{online_dispatch}), the problem setting and the proposed distributed online algorithm studied in this paper are applicable to a broad range of practical scenarios. Representative examples include PEV charging scheduling problem \citep{Vujanic2016,LiJ2018}, online linear regression problem \citep{Yi2025}, and other related distributed optimization tasks.

    \begin{example}
    PEV problem can be formulated as
    \begin{equation}\nonumber 
			\begin{aligned}
				\min _{x_i\in \Omega_i } &\sum_{t=1}^{T}\sum_{i=1}^N \frac{a_{i,t}}{2}\|x_i\|^2 +\beta_{i,t}^{\top} x_i,\\
				\text{s.t.}\,\, &  \sum_{i=1}^N A_{i,t} x_i -d_{i,t} \leq \mathbf{0}_n, \,\,  t=1,\dots,T,
			\end{aligned}
    \end{equation}
    where $x_{i,t}$ is the charging rate of vehicle $i$, and $a_{i,t}, \beta_{i,t}, A_{i,t}, d_{i,t}$ are time‑varying, privately held coefficients.
    \end{example}

    \begin{example}
    The online linear regression problem with time‑varying nonlinear inequality constraints is given as
    \begin{equation}\nonumber
    \begin{aligned}
    \min_{x\in\mathcal{X}} &\quad \sum_{t=1}^{T}\sum_{i=1}^n \tfrac{1}{2}\|H_{i,t} x - h_{i,t}\|^2,\\
    \text{s.t.} &\quad b_{i,t} - \log\big(1+\|x\|^2\big) \le 0,\qquad \forall i\in[n],\ \forall t\in[T],
    \end{aligned}
    \end{equation}
    where $x\in\mathbb{R}^p$ is the common decision vector, $H_{i,t}\in\mathbb{R}^{d_i\times p}$, $h_{i,t}\in\mathbb{R}^{d_i}$, $b_{i,t}\in\mathbb{R}$, and $d_i\in\mathbb{N}_+$ represents the output/response dimension of node $i$.
    \end{example}
    
    Similar formulations also arise in energy storage management, renewable energy integration, and resource allocation problems \citep{wu2005,Nasrabadi2012,ahn2024data}. These settings typically involve local decision variables, shared coupling constraints, and time-varying parameters, making them well suited to distributed online optimization approaches.

    \subsection{Organization and notation}
	
    The remainder of this paper is organized as follows. Section 2 presents preliminaries and the problem formulation. Section 3 introduces the design of the proposed algorithm. Section 4 analyzes a theoretical analysis of the algorithm's dynamic regret and constraint violation. Section 5 illustrates the algorithm's performance through numerical experiments using both real and simulated datasets. Finally, Section 6 concludes the paper.
	
	$\mathbb{R}$ and $\mathbb{Z}$ denote the sets of real and integer numbers, respectively; $\mathbb{R}^n$ is the $n$-dimensional Euclidean space. For a vector or matrix $X$, $\Vert X \Vert$, $\Vert X \Vert_1$, $\Vert X \Vert_{\infty}$ denote its 2-norm, 1-norm, and infinity-norm, respectively; $X^{\top}$ denotes the transpose of $X$; $[X]_i$ denotes the $i$-th row if $X$ is a matrix or the $i$-th component if $X$ is a vector. $\nabla f$ denotes the gradient of the function $f$. $\textbf{1}$ denotes a column vector of appropriate dimension with all entries equal to $1$; $I$ denotes the identity matrix of appropriate dimension;  $\textbf{0}_n$ denotes the $n$‑dimensional zero vector. For $x \in \mathbb{R}^n,\, [x]_+= \max\{\textbf{0}_n, x \}$, where $\max$ is taken element-wise. For a closed convex set $\Omega$, $P_{\Omega}(x)$ is the projection of $x$ on $\Omega$, i.e,, $P_{\Omega}(x)= \arg\min_{y\in \Omega} \|x-y\|^2$.

	\section{Preliminaries and problem formulation }
	
	This section presents the preliminaries and then formulates the problem of interest.

	\subsection{Graph theory}
	
	Let the communication graph be $G(t)=\{V,E(t),A(t)\}$, $t=1, 2,\dots, T$, where $V=\{1,2,\dots,$ $N\}$ is the set of nodes and $E(t)$ is the edge set. An element of $E(t)$ is an ordered pair $(j,i)$ with $j,i\in V$, and $(j,i)\in E(t)$ denotes node $j$ can directly transmit information to node $i$. The in- and out-neighborhoods of node $i$ at time $t$ are denoted by $N_i^{in}(t)=\{j|(j,i)\in E(t)\}\cup \{i\}$ and $N_i^{out}(t)=\{j|(i,j)\in E(t)\}\cup \{i\}$, respectively. The matrix $A(t)=[a_{ij}(t)]_{N \times N}$ is called the weighted adjacency matrix of $G(t)$. Its entries satisfy $a<a_{ij}(t) <1$ if $(j,i)\in E(t)$, for some constant $0<a<1$, and $a_{ij}(t) = 0$ otherwise. If $A(t)$ is symmetric, then $G(t)$ is an undirected graph. Moreover, if $\boldsymbol{1}^{\top} A(t)=\boldsymbol{1}^{\top}$, and $A^{\top}(t) \boldsymbol{1}= \boldsymbol{1}$, then $A(t)$ is a doubly stochastic matrix and $G(t)$ is a balanced graph.
    
	For a fixed graph $G=\{V,E,A\}$, $G$ is strongly connected if, for any two vertices $i,j\in V$, there exists a sequence of nodes $k_1,k_2,\dots,k_m\in V$ such that $(i,k_1),(k_1,k_2),\dots,(k_m,j)\in E$. 
	
    For a time-varying graph $G(t)$, the $B$-edge set is defined as $E_B(t)=\cup_{k=(t-1)B+1}^{tB} E(k)$, where $B \geq 1$ is an integer. For any $t\geq 1$, if $\{V, E_B(t)\}$ is strongly connected, then $G(t)$ is $B$-strongly connected.

	\subsection{Problem formulation}

    Generally, we consider a multi-agent system with $N$ nodes, where each agent only knows its own $f_{i,t}$ and $g_{i,t}$, and communicates with its neighbors over a directed time-varying network $G(t)$. At each time $t$, the agents cooperatively solve the following optimization problem:
		\begin{equation}\label{constarinedproblem}
		\begin{aligned}
		\min _{x_i\in \Omega_i } f_t(x)= &\sum_{i=1}^N f_{i,t}(x_i),\\
			\text{s.t.}\,\, g_t(x)= &  \sum_{i=1}^N g_{i,t}(x_i) \leq \mathbf{0}_n, \,\,  t=1,\dots,T,
		\end{aligned}
	\end{equation}
	where $f_{i,t}: \Omega_i  \rightarrow \mathbb{R}$ is the local objective function of node $i$ at time $t$, $g_{i,t}: \Omega_i  \rightarrow \mathbb{R}^n$ is the local constraint function of node $i$ at time $t$, $x_i \in \Omega_i \in \mathbb{R}^{d_i}$ is the decision variable of node $i$, $d_{i}$ denotes the dimension of the decision variable of node $i$, $x=(x_1^{\top}, \dots, x_N^{\top})^{\top} \in \mathbb{R}^d$,
	$\sum_{i=1}^N d_i= d$, and $T$ is the time horizon. At each time $t$, node $i$ only knows $f_{i,t}$, $g_{i,t}$  and $\Omega_i$. Let $\Omega = \Omega_1 \times \dots \times \Omega_N$. The feasible set of problem (\ref{constarinedproblem}) at time $t$ is $\bar{\Omega}_t=\{x \in \Omega | g_t(x) \leq \mathbf{0}_n\}$. 

    \begin{remark}
    The formulation \eqref{constarinedproblem} is general and accommodates time-varying coupled inequality constraints of the form $g_t(x)=\sum_{i=1}^N g_{i,t}(x_i)\le \mathbf{0}_n$ together with local set constraints.
    (i) If the inequality constraints are decoupled, i.e., $g_{i,t}(x_i) \leq 0, \, i=1,\dots, N$, then the optimization problem can be solved independently by each agent without any information exchange. This corresponds to a trivial special case.
    (ii) If the constraints are time-invariant, i.e., $g_{t}(x)= g(x)$, then the problem reduces to the well‑studied setting of distributed (online) optimization with static coupling constraints. However, real-world operating environments are typically time-varying. Accounting for time-varying constraints is therefore more realistic and also more technically challenging, as it requires algorithms that can adapt to nonstationarity and provide dynamic performance guarantees. (iii) When the inequality constraints become linear, we can prove the equivalence of their solution to that of the equality constraints. Thus, under appropriate conditions, our problem can cover scenarios involving equality constraints.
    \end{remark}

    For problem (\ref{constarinedproblem}), we give the following assumptions.
	
	\begin{assumption}\label{AS4.1}
		For any $t=1,\dots,T$, $G(t)$ is $B$-strongly connected and its weight matrix $A(t)$ is doubly stochastic.
	\end{assumption}
	
	\begin{assumption}\label{AS4.2}
		The set $\Omega \subseteq \mathbb{R}^d$ is a bounded, closed, and convex set, i.e., there exists a constant $\eta>0$ such that $\|x\| \leq \eta$, $\forall x\in \Omega$.
	\end{assumption}
	
	\begin{assumption}\label{AS4.3}
		For all $t=1,\dots,T$ and $i=1,\dots,N$, $f_{i,t}(x)$ and $g_{i,t}(x)$ are convex and differentiable with respect to $x$, and there exist constants $C_1,C_2>0$ such that for any $x_i \in \Omega_i$, $|f_{i,t}(x_i)|\leq C_1$, $\|g_{i,t}(x_i)\|\leq C_1$, $\|\nabla f_{i,t}(x_i)\|\leq C_2$, $\|\nabla g_{i,t}(x_i)\|\leq C_2$, where $\nabla g_{i,t}(x_i)$ denotes the $n \times d_i$ Jacobian matrix.
	\end{assumption}
	
	Assumptions \ref{AS4.2}-\ref{AS4.3} are standard in the literature on distributed optimization \citep{zhangy2019, zhouyingjie2024}. Since $f_{i,t}(x)$ and $g_{i,t}(x)$ are differentiable on a bounded, closed domain, they are necessarily bounded on the domain. Compared with assuming $L$-smoothness of the gradients, we only assume bounded gradients, which is a weaker condition.

    \begin{remark}
    Assumption \ref{AS4.1} ensures sufficient information mixing so that local quantities average out and consensus errors can be bounded. Assumption \ref{AS4.2} guarantees well‑defined projections and bounded iterates. Assumption \ref{AS4.3} ensures any local optimum is global. As an illustrative example, consider problem (\ref{online_dispatch}), Assumption \ref{AS4.1} implies that information can be exchanged among generators and that the communication network contains no isolated nodes. Assumption \ref{AS4.2} corresponds to the practical limitation on generation capacities. The local objective and constraint functions satisfy Assumption \ref{AS4.3}.
    \end{remark}

	At time $t$, node $i$ can only observe $f_{i,t}$ and $g_{i,t}$, to solve the global optimization task, nodes must communicate via the network $G(t)$. In distributed online convex optimization, each node uses its local information and the information received from its neighbors to make a local prediction $x_{i,t+1}$ of the optimal solution $x^*_{i,t+1}$ at time $t+1$, where $x_{i,t+1}$ is the decision variable of node $i$ at time $t+1$. The performance of the algorithm is evaluated in terms of dynamic regret, which is defined as follows.
	\begin{definition}\label{dynamci_regret} \citep{Shahrampour2018}
		If $\{x_{i,t}, t=1,...,T, i=1,...,N \}$ are the decision variables generated by a given distributed online optimization algorithm, then the algorithm's dynamic regret is defined as
		$$Reg_T^D \triangleq \sum_{t=1}^{T}\sum_{i=1}^N\left(f_{i,t}\left(x_{i,t}\right)-f_{i,t}\left(x_{i,t}^*\right)\right),$$
		where $x_t^*=\arg\min_{x\in \bar{\Omega}_t}f_t\left(x\right)$.
	\end{definition}
	
	The dynamic regret depends not only on the running time $T$ but also on the variation path of the algorithm's optimal solutions, which we denote by $P_T=\sum_{t=2}^{T} \sum_{i=1}^N\| x_{i,t}^*-x_{i,t-1}^*\|.$ Since problem (\ref{constarinedproblem}) considers time-varying coupled inequality constraints, it is also important to study the algorithm's constraint violation, which is defined as follows.
	\begin{definition}\label{violation} \citep{Lijueyou2022,Lixiuxian2020}
		If $\{x_{i,t}, t=1,...,T, i=1,...,N \}$ are the decision variables generated by a given distributed online optimization algorithm, then the algorithm's constraint violation is defined as
		$$Vio_T= \left\|\left[\sum_{t=1}^{T} g_t(x) \right]_{+} \right\|.$$
	\end{definition}
	The goal of this paper is to design an appropriate algorithm such that both the dynamic regret and the constraint violation grow sublinearly, i.e., $\lim_{T\rightarrow \infty} Reg_T^D/T=0$ and $\lim_{T\rightarrow \infty} Vio_T/T=0$.
	
	\section{Algorithm design}
	
	Firstly, we present the centralized primal-dual method for solving problem (\ref{constarinedproblem}). At each time $t$, the Lagrangian function with a regularization term is defined as
	\begin{equation}\label{regularized_LF}
		L_t(x,\lambda)=\sum_{i=1}^N f_{i,t}(x_i) +\lambda^{\top} \sum_{i=1}^N g_{i,t}(x_i)-\frac{\gamma_t}{2}\|\lambda\|^2,
	\end{equation}
	where $\lambda \in \mathbb{R}^n_+$ is the dual variable and $\gamma_t>0$ is the regularization parameter which is assumed here to be non-increasing over time. Compared with the standard Lagrangian function, (\ref{regularized_LF}) contains an additional regularization term to prevent the dual variable from growing unbounded. We aim to find a saddle point of $L_t(x, \lambda)$, i.e., to solve the problem $\max_{\lambda \geq 0}\min_{x\in \Omega}L_t(x, \lambda)$. Using the centralized online primal-dual method, the updates are given as follows:
	\begin{align}
		x_{i,t+1}&= P_{\Omega_i}[x_{i,t} - \alpha_t\nabla_{x_i}L_t(x_t,\lambda_t)], \,\, i=1,\dots, N, \label{centra_al1}\\
		\lambda_{t+1}& = [\lambda_t + \alpha_t \nabla_{\lambda}L_t(x_t, \lambda_t)]_+,  \label{centra_al2}
	\end{align}
	where $\alpha_t>0$ is the step size. Note that $\nabla_{x_i}L_t(x_t,\lambda_t)=\nabla f_{i,t}(x_{i,t}) + \nabla g_{i,t}(x_{i,t})^{\top}\lambda_t$, $\nabla_{\lambda}L_t(x_t, \lambda_t)=\sum_{i=1}^N g_{i,t}(x_i) -\gamma_t \lambda_t$.
    
    In (\ref{centra_al1})-(\ref{centra_al2}), each agent requires access to global information when updating its decision, which is generally unavailable in a distributed setting. To enable distributed implementation, let $\lambda_{i,t}$ denote the node $i$'s local estimate of the dual variable at time $t$. Then $\lambda_{i,t+1}$ is updated as 
	$$\lambda_{i,t+1}= \left[ \lambda_{i,t} +\alpha_t(g_{i,t}(x_{i,t})-\gamma_t \lambda_{i,t}) \right]_+.$$
	To make each local $\lambda_{i,t}$ a more accurate estimate of the global $\lambda_t$, each node needs to estimate the aggregate constraint $\sum_{i=1}^N g_{i,t}(x_{i,t})$. Motivated by gradient-tracking algorithms \citep{pushi2021,sun2022}, where each node uses its local gradient to estimate the global gradient, we extend this idea to allow each node to obtain an estimate of $\sum_{i=1}^N g_{i,t}(x_{i,t})$.
    
    The resulting distributed online optimization algorithm for economic dispatch with coupled inequality constraints is summarized below. 
    
    At each time $t$, agent $i$ first exchanges its local dual estimate $\lambda_{i,t}$ and and constraint-tracking variable $y_{i,t}$ with its neighbors through the weighted matrix $A(t)$, and computes the local averages $\mu_{i,t}$ and $z_{i,t}$. Provided that the initialization satisfies $y_{i,1}= Ng_{i,1}(x_{i,1})$ for all $i=1,\dots,N$, each agent can obtain an accurate estimate of the aggregate constraint over time. Notably, only the dual and constraint-related variables $\lambda_{i,t}$ and $y_{i,t}$ are exchanged, while the primal decision variables remain private, thereby preserving agent privacy and ensuring a fully distributed implementation.
    
    Using these local estimates, agent $i$ updates its decision via a projected gradient step on the local Lagrangian
    $$x_{i,t+1}=P_{\Omega_i}\big[x_{i,t}-\alpha_t\big(\nabla f_{i,t}(x_{i,t})+\nabla g_{i,t}(x_{i,t})^\top \mu_{i,t}\big)\big].$$
    The dual variable is then updated through a projected ascent step,
    $$\lambda_{i,t+1}=\big[\mu_{i,t}+\alpha_t\big(z_{i,t}-\gamma_t\mu_{i,t}\big)\big]_+,$$
    followed by an update of the constraint-tracking variable by injecting the local constraint increment 
    $$y_{i,t+1}=z_{i,t}+N\big(g_{i,t}(x_{i,t+1})-g_{i,t}(x_{i,t})\big).$$
    
    The complete procedure of the proposed algorithm is presented in Algorithm \ref{algo2}.
	
		\begin{algorithm}[H]
		\begin{algorithmic}[1]%每行显示行号
			\caption{Distributed online economic dispatch algorithm with coupled inequality constraints.} \label{algo2}
			\State Input: non-incresing step size $\alpha_t$ and regularization parameter $\gamma_t$, initial value $x_{i,1}\in \Omega_i $,  $\lambda_{i,1}=\mathbf{0}_n$, $y_{i,1}=Ng_{i,1}(x_{i,1})$, $i=1,\dots,N$, weighted matrix $A(t)$.
			\For{$t = 1, 2, \dots, T$}
			\For{$i = 1, 2, \dots, N$}
			\State $\mu_{i,t}= \sum_{j=1}^N a_{ij}(t)\lambda_{j,t}$
			\State $z_{i,t} =\sum_{j=1}^N a_{ij}(t) y_{j,t}$
			\State $x_{i,t+1} = P_{\Omega_i}\left[x_{i,t}- \alpha_t\left(\nabla f_{i,t}(x_{i,t}) + \nabla g_{i,t}(x_{i,t})^{\top}\mu_{i,t} \right) \right]$
			\State $\lambda_{i,t+1}= \left[\mu_{i,t}+ \alpha_t\left(z_{i,t} -\gamma_t\mu_{i,t}\right) \right]_+$
			\State $y_{i,t+1} = z_{i,t}+N[g_{i,t+1}(x_{i,t+1})- g_{i,t}(x_{i,t})]$
			\EndFor
			\EndFor
			\State Output: $\{x_{t}^i\}_{i=1}^N$
		\end{algorithmic}
	\end{algorithm}

    \begin{remark}
    As shown in Algorithm \ref{algo2}, the decision updates rely only on information from the current and immediately preceding time steps and do not require storing the full history. Consequently, the proposed algorithm offers substantial gains in computational efficiency and memory usage.
    \end{remark}

	\section{Theoretical analysis}
	Let $\bar{\lambda}_t =\frac{1}{N}\sum_{i=1}^N \lambda_{i,t}$, $\bar{y}_t=\frac{1}{N} \sum_{i=1}^N y_{i,t}$, $x_t=[x_{1,t}^{\top}, \dots, x_{N,t}^{\top}]^{\top}$. For any $t \geq s \geq 1$, the state transition matrix is defined as
	$$
	\Phi(t, s)=\left\{\begin{array}{l}
		A(t-1) \cdots A(s+1) A(s), \text { if }\,\, t>s, \\
		I_N, \text { if }\,\, t=s.
	\end{array}\right.
	$$
	We first give the following lemmas.
	\begin{lemma}\label{gt_tracking}
		Suppose Assumption \ref{AS4.1} holds. For the sequences $\{\mu_{i,t}\}$ and $\{z_{i,t}\}$ generated by Algorithm \ref{algo2}, $t = 1, \dots, T$, $i=1,\dots, N$, we have   
		\begin{align}
			\bar{\lambda}_t &= \frac{1}{N}\sum_{i=1}^N\mu_{i,t}, \label{lambda_ave}\\
			\bar{y}_t &= \frac{1}{N} \sum_{i=1}^N z_{i,t} =\sum_{i=1}^N g_{i,t}(x_{i,t}). \label{yt_ave}
		\end{align}
	\end{lemma}
	\begin{proof}
		By Assumption \ref{AS4.1}, we have
		\begin{align*}
			\frac{1}{N}\sum_{i=1}^N\mu_{i,t} &= \frac{1}{N}\sum_{i=1}^N\sum_{j=1}^N a_{ij}(t)\lambda_{j,t}\\
			&=\frac{1}{N} \sum_{j=1}^N \left(\sum_{i=1}^Na_{ij}(t) \right)\lambda_{j,t}\\
			&=\frac{1}{N} \sum_{j=1}^N \lambda_{j,t} = \bar{\lambda}_t.
		\end{align*}
		Similarly, we have $\bar{y}_t =\frac{1}{N} \sum_{i=1}^N z_{i,t}$. It follows from Algorithm \ref{algo2} that
		\begin{align*}
			\bar{y}_t &= \frac{1}{N} \sum_{i=1}^N \left( z_{i,t-1}+N[g_{i,t}(x_{i,t})- g_{i,t-1}(x_{i,t-1})] \right)\\
			&= \frac{1}{N} \sum_{i=1}^N z_{i,t-1} +\sum_{i=1}^N g_{i,t}(x_{i,t})- \sum_{i=1}^N g_{i,t-1}(x_{i,t-1})\\
			&= \bar{y}_{t-1} +\sum_{i=1}^N g_{i,t}(x_{i,t})- \sum_{i=1}^N g_{i,t-1}(x_{i,t-1}).
		\end{align*}
		Therefore, we have
		$$\bar{y}_t - \sum_{i=1}^N g_{i,t}(x_{i,t})= \bar{y}_{t-1}- \sum_{i=1}^N g_{i,t-1}(x_{i,t-1}) =\dots =\bar{y}_1-\sum_{i=1}^N g_{i,1}(x_{i,1}).$$
		Combining with $y_{i,1}=Ng_{i,1}(x_{i,1})$, we obtain $\bar{y}_t = \sum_{i=1}^N g_{i,t}(x_{i,t})$.
	\end{proof}
	
	From the above lemma, we know that if the initial variables $y_{i,1}$ satisfy the appropriate conditions, then $\bar{y}_t$ equals the global constraint function at time $t$. Therefore, in the algorithm, each agent uses $z_{i,t}$ to track $\bar{y}_t$, so that its local variable can use the global coupled constraint information.
	
	\begin{lemma}\label{graph_lemma}
		Suppose Assumption \ref{AS4.1} holds. For any $i, j \in V$, $t \geq s$, we have
		\begin{equation}
			\qquad\left|[\Phi(t, s)]_{i j}-\frac{1}{N}\right| \leq \hat{C} \tau^{t-s},
		\end{equation}
		where $\hat{C}=2\left(1+a^{-(N-1)B}\right) /\left(1+a^{(N-1)B}\right) >1$ and $\tau=  \left(1-a^{(N-1)B}\right)^{1 /(N-1)B} \in (0,1)$.
	\end{lemma}
	
	\begin{lemma}\label{y_bounded}
		Suppose Assumptions \ref{AS4.1}-\ref{AS4.3} hold. For the sequences $\{y_{i,t}\}$ and $\{z_{i,t}\}$
		generated by Algorithm \ref{algo2}, $t = 1, \dots, T$, $i=1,\dots, N$, there exists a constant $C>0$ such that $\|z_{i,t}\| \leq C$ and $\|y_{i,t}\| \leq C$. 
	\end{lemma}
	\begin{proof}
		Let $\epsilon_{i,t-1}^y = y_{i,t}-z_{i,t-1}= N \left[g_{i,t}(x_{i,t})-g_{i,t-1}(x_{i,t-1}) \right]$. By Assumption \ref{AS4.3}, we have $\|\epsilon_{i,t-1}^y\| \leq N\|g_{i,t}(x_{i,t})-g_{i,t-1}(x_{i,t-1}) \| \leq 2NC_1$. By the doubly stochasticity of the weight matrix $A(t)$, we have
		\begin{align*}
			y_{i,t} &= z_{i,t-1} + \epsilon_{i,t-1}^y\\
			&= \sum_{j=1}^N a_{ij}(t-1) y_{j,t-1} +  \epsilon_{i,t-1}^y\\
			&= \sum_{j=1}^N [\Phi(t-1, 1)]_{i j} y_{j,1} + \sum_{s=1}^{t-2} \sum_{j=1}^N [\Phi(t-1, s+1)]_{i j} \epsilon_{j,s}^y + \epsilon_{i,t-1}^y,
		\end{align*}
		and
		\begin{align*}
			\bar{y}_t &= \frac{1}{N} \sum_{i=1}^N z_{i,t-1} + \frac{1}{N} \sum_{i=1}^N \epsilon_{i,t-1}^y\\
			&= \frac{1}{N} \sum_{j=1}^N y_{j,1} + \frac{1}{N} \sum_{s=1}^{t-2} \sum_{j=1}^N \epsilon_{j,s}^y  + \frac{1}{N} \sum_{j=1}^N \epsilon_{j,t-1}^y.
		\end{align*}
		Then we have
		\begin{align*}
			\|y_{i,t} - \bar{y}_t\|  \leq &\sum_{j=1}^N \left|[\Phi(t-1, 1)]_{i j}-\frac{1}{N} \right| \|y_{j,1}\| +\sum_{s=1}^{t-2} \sum_{j=1}^N \left| [\Phi(t-1, s+1)]_{i j}- \frac{1}{N} \right| \|\epsilon_{j,s}^y\|\\
			& + \|\epsilon_{i,t-1}^y\| + \frac{1}{N} \sum_{j=1}^N \|\epsilon_{j,t-1}^y\| \\
			\leq &\hat{C}\tau^{t-2}\sum_{j=1}^N \|y_{j,1}\| + \hat{C} \sum_{s=1}^{t-2} \tau^{t-s-2} \sum_{j=1}^N \|\epsilon_{j,s}^y\|+ \|\epsilon_{i,t-1}^y\| + \frac{1}{N} \sum_{j=1}^N \|\epsilon_{j,t-1}^y\| \\
			\leq & N\hat{C}\tau^{t-2}\max_{j}\|y_{j,1}\| + \frac{2N^2\tau C_1\hat{C}}{1-\tau}+2NC_1.
		\end{align*}
		Combining the above equation and $\|\bar{y}_t\|= \left\|\sum_{i=1}^N g_{i,t}(x_{i,t})\right\| \leq \sum_{i=1}^N\|g_{i,t}(x_{i,t})\|\leq NC_1$ yields
		\begin{align*}
			\|y_{i,t}\| &\leq \|y_{i,t} - \bar{y}_t\| + \|\bar{y}_t\|\\
			&\leq N\hat{C}\tau^{t-2}\max_{j}\|y_{j,1}\| + \frac{2N^2 \tau C_1\hat{C}}{1-\tau}+3NC_1\\
			& \triangleq C.
		\end{align*}
		We also have
		$$\|z_{i,t}\| =\left\|\sum_{j=1}^N a_{ij}(t) y_{j,t} \right\| \leq \sum_{j=1}^N a_{ij}(t) \left\|y_{i,t}\right\| \leq C.$$
	\end{proof}
	
	\begin{lemma}\label{lambda_bounded}
		Suppose Assumptions \ref{AS4.1}-\ref{AS4.3} hold. For the sequences $\{\mu_{i,t}\}$ and $\{\lambda_{i,t}\}$
		generated by Algorithm \ref{algo2}, $t = 1, \dots, T$, $i=1,\dots, N$, if $\alpha_t \gamma_t \leq 1$, then we have
		 $\|\mu_{i,t}\| \leq \frac{C}{\gamma_t}$ and $\|\lambda_{i,t}\| \leq \frac{C}{\gamma_t}$.  
	\end{lemma}
	
	\begin{proof}
		We prove this result by mathematical induction. If $t=1$, then we have $\|\lambda_{i,1}\|=0$, $\|\mu_{i,1}\|= \left\|\sum_{j=1}^N a_{ij}(1)\lambda_{j,1}\right\|=0$, $\|\lambda_{i,1}\| \leq \frac{C}{\gamma_1}$, and $\|\mu_{i,1}\| \leq \frac{C}{\gamma_1}$. Assume that it holds for $t$, i.e., $\|\mu_{i,t}\| \leq \frac{C}{\gamma_t}$, $\|\lambda_{i,t}\| \leq \frac{C}{\gamma_t}$, then we have
		\begin{align*}
		\|\mu_{i,t}- \alpha_t\left(z_{i,t} -\gamma_t\mu_{i,t}\right)\| &\leq (1-\alpha_t\gamma_t)\|\mu_{i,t}\|+ \alpha_t\|z_{i,t}\|\\
		&\leq (1-\alpha_t \gamma_t)\frac{C}{\gamma_t}+\alpha_t C\\
		&= \frac{C}{\gamma_t}.
		\end{align*}
		Since $\gamma_t$ is non-increasing, it follows that
		$$\|\lambda_{i,t+1}\|=\|\left[\mu_{i,t}- \alpha_t\left(z_{i,t} -\gamma_t\mu_{i,t}\right) \right]_+\| \leq \frac{C}{\gamma_t} \leq \frac{C}{\gamma_{t+1}}.$$
		By Assumption \ref{AS4.1}, we have
		$$\|\mu_{i,t+1}\|= \left\|\sum_{j=1}^N a_{ij}(t) \lambda_{j,t+1} \right\| \leq \sum_{j=1}^N a_{ij}(t) \|\lambda_{j,t+1}\| \leq \frac{C}{\gamma_{t+1}}.$$
	\end{proof}
	
	\begin{remark}
		From Lemma \ref{lambda_bounded}, it follows that the Lagrange multipliers in Algorithm \ref{algo2} can be bounded by the coefficients of the regularized Lagrangian function. Hence, in the subsequent analysis of this paper, we do not need to assume the boundedness of the Lagrange multipliers, an assumption required in \cite{Luk2023, Ayslender2000}.
	\end{remark}
	
	\begin{lemma}\label{lam_consensus}
		Suppose Assumptions \ref{AS4.1}-\ref{AS4.3} hold. For the sequences $\{\mu_{i,t}\}$ and $\{\lambda_{i,t}\}$ generated by Algorithm \ref{algo2}, $t = 1, \dots, T$, $i=1,\dots, N$, we have
		$$\sum_{t=1}^{T} \sum_{i=1}^N \|\mu_{i,t}- \bar{\lambda}_t\| \leq  \sum_{t=1}^{T} \sum_{i=1}^N \|\lambda_{i,t}- \bar{\lambda}_t\| \leq D\sum_{t=1}^{T} \alpha_t,$$
		where $D=\frac{2N^2\hat{C}C}{1-\tau}+2NC$.
	\end{lemma}
	\begin{proof}
		By Assumption \ref{AS4.1}, we have
		\begin{align*}
			\sum_{t=1}^{T} \sum_{i=1}^N \|\mu_{i,t}- \bar{\lambda}_t\| &= \sum_{t=1}^{T} \sum_{i=1}^N  \left\|\sum_{j=1}^N a_{ij}(t)\lambda_{j,t}- \bar{\lambda}_t \right\|\\
			&\leq \sum_{t=1}^{T} \sum_{i=1}^N \sum_{j=1}^N a_{ij}(t) \|\lambda_{j,t} - \bar{\lambda}_t\|\\
			&= \sum_{t=1}^{T} \sum_{i=1}^N \|\lambda_{i,t}- \bar{\lambda}_t\|.
		\end{align*}
		By Lemmas \ref{y_bounded}-\ref{lambda_bounded}, it follows that
		$$\|z_{i,t} -\gamma_t\mu_{i,t}\| \leq \|z_{i,t}\|+\gamma_t\|\mu_{i,t}\|\leq C+\gamma_t \frac{C}{\gamma_t}=2C.$$
		Let $\lambda_{i,t}=\mu_{i,t-1}+\epsilon_{i,t-1}^{\lambda}$. By a similar proof of Lemma \ref{yt_ave}, we have
		$$\|\epsilon_{i,t-1}^{\lambda}\|=\|\left[\mu_{i,t-1}+ \alpha_{t-1}\left(z_{i,t-1} -\gamma_{t-1}\mu_{i,t-1}\right) \right]_+-\mu_{i,t-1}\| \leq 2\alpha_{t-1}C$$
		By Assumption \ref{AS4.1} and Lemma \ref{gt_tracking}, we have 
	\begin{align*}
		\|\lambda_{i,t+1} - \bar{\lambda}_{t+1}\|  \leq &\sum_{j=1}^N \left|[\Phi(t, 1)]_{i j}-\frac{1}{N} \right| \|\lambda_{j,1}\| +\sum_{s=1}^{t-1} \sum_{j=1}^N \left| [\Phi(t, s+1)]_{i j}- \frac{1}{N} \right| \|\epsilon_{j,s}^{\lambda}\|\\
		& + \|\epsilon_{i,t}^{\lambda}\| + \frac{1}{N} \sum_{j=1}^N \|\epsilon_{j,t}^{\lambda}\| \\
		\leq &\hat{C}\tau^{t-1}\sum_{j=1}^N \|\lambda_{j,1}\| + \hat{C} \sum_{s=1}^{t-1} \tau^{t-s-1} \sum_{j=1}^N \|\epsilon_{j,s}^{\lambda}\|+ \|\epsilon_{i,t}^{\lambda}\| + \frac{1}{N} \sum_{j=1}^N \|\epsilon_{j,t}^{\lambda}\| \\
		\leq & \left(\frac{2N\tau C \hat{C}}{1-\tau}+2C\right)\alpha_{t},
	\end{align*}
	where the last inequality holds by $\lambda_{i,1}=0$. Hence, we have
\begin{align*}
	\sum_{t=1}^{T-1} \sum_{i=1}^N \|\lambda_{i,t+1} - \bar{\lambda}_{t+1}\| \leq D\sum_{t=1}^{T-1}\alpha_t.
\end{align*}
Then we have
$$\sum_{t=1}^{T} \sum_{i=1}^N \|\lambda_{i,t} - \bar{\lambda}_{t}\| \leq \sum_{i=1}^N \|\lambda_{i,1} - \bar{\lambda}_{1}\| + D\sum_{t=1}^{T}\alpha_t= D\sum_{t=1}^{T}\alpha_t.$$
	\end{proof}
	
	\begin{lemma}\label{lemma_Lt_first}
		Suppose Assumptions \ref{AS4.1}-\ref{AS4.3} hold. For the sequences $\{x_{i,t}\}$ and $\{\lambda_{i,t}\}$ generated by  Algorithm \ref{algo2}, $t = 1, \dots, T$, $i=1,\dots, N$, we have 
		\begin{align}
			L_t(x_t, \bar{\lambda}_t)-	L_t(x_t^*, \bar{\lambda}_t) \leq & \frac{1}{2\alpha_t}\sum_{i=1}^N \left(\|x_{i,t}-x_{i,t}^*\|^2 -\|x_{i,t+1}-x_{i,t+1}^*\|^2 \right) + 2\eta C_2 \sum_{i=1}^N \|\mu_{i,t}-\bar{\lambda}_t\| \nonumber \\
			&+ \sum_{i=1}^N \left(\frac{3\eta}{\alpha_t}+ C_2\left(1+ \|\mu_{i,t}\|  \right) \right)\|x_{i,t}^* - x_{i,t+1}^*\|+\alpha_t C_2^2\sum_{i=1}^N\left(1+ \|\mu_{i,t}\|^2\right),\label{lemma_Lt_first_eq1}
		\end{align}
		where $x_t^*= [(x_{1,t}^*)^{\top}, \dots, (x_{N,t}^*)^{\top}]^{\top}=\arg\min_{x\in \bar{\Omega}_t}f_t\left(x\right)$.
	\end{lemma}
	\begin{proof}
		By Assumption \ref{AS4.3} and Lemma \ref{lambda_bounded}, we have 
		$$\|\nabla f_{i,t}(x_{i,t}) + \nabla g_{i,t}(x_{i,t})^{\top}\mu_{i,t}\|\leq \|\nabla f_{i,t}(x_{i,t})\| +\|\nabla g_{i,t}(x_{i,t})\|\|\mu_{i,t}\|\leq C_2\left(1+ \|\mu_{i,t}\|\right),$$
		$$\|\nabla f_{i,t}(x_{i,t}) + \nabla g_{i,t}(x_{i,t})^{\top}\mu_{i,t} \|^2 \leq 2\|\nabla f_{i,t}(x_{i,t})\|^2 + 2\|\nabla g_{i,t}(x_{i,t})\|^2 \|\mu_{i,t} \|^2 \leq 2C_2^2\left(1+ \|\mu_{i,t}\|^2\right).$$
		Combining the above equation with Assumption \ref{AS4.2}, we obtain
		\begin{align}
			\|x_{i,t+1}-x_{i,t+1}^*\|^2 \leq & \|x_{i,t}-x_{i,t+1}^* -\alpha_t\left(\nabla f_{i,t}(x_{i,t}) + \nabla g_{i,t}(x_{i,t})^{\top}\mu_{i,t} \right)\|^2 \nonumber\\
			=&\|x_{i,t}-x_{i,t+1}^*\|^2-2\alpha_t \langle  x_{i,t}- x_{i,t+1}^*, \nabla f_{i,t}(x_{i,t}) + \nabla g_{i,t}(x_{i,t})^{\top}\mu_{i,t} \rangle \nonumber\\
			&+\alpha_t^2 \|\nabla f_{i,t}(x_{i,t}) + \nabla g_{i,t}(x_{i,t})^{\top}\mu_{i,t} \|^2 \nonumber\\
			\leq & \|x_{i,t}-x_{i,t}^*\|^2 +6\eta \|x_{i,t}^*-x_{i,t+1}^*\|+2\alpha_t C_2\left(1+ \|\mu_{i,t}\|\right)\|x_{i,t}^*-x_{i,t+1}^*\| \nonumber\\
			&+2\alpha_t^2 C_2^2\left(1+ \|\mu_{i,t}\|^2\right)-2\alpha_t \langle  x_{i,t}- x_{i,t}^*, \nabla f_{i,t}(x_{i,t}) + \nabla g_{i,t}(x_{i,t})^{\top}\mu_{i,t} \rangle \label{lemma_Lt_first_eq2}.
		\end{align}
		Since $L_t(x,\lambda)$ is convex with respect to $x$, we have 
		\begin{align}
			&\sum_{i=1}^N \langle  x_{i,t}- x_{i,t}^*, \nabla f_{i,t}(x_{i,t}) + \nabla g_{i,t}(x_{i,t})^{\top}\mu_{i,t} \rangle \nonumber \\
			=& \sum_{i=1}^N \left[ \nabla f_{i,t}(x_{i,t}) + \nabla g_{i,t}(x_{i,t})^{\top}\bar{\lambda}_t  \right]^{\top}(x_{i,t}- x_{i,t}^*) +\sum_{i=1}^N \left[ \nabla g_{i,t}(x_{i,t})^{\top}(\mu_{i,t}-\bar{\lambda}_t)  \right]^{\top}(x_{i,t}- x_{i,t}^*) \nonumber \\
			\geq & \sum_{i=1}^N \nabla_{x_i}L_t(x_{i,t},\bar{\lambda}_t)^{\top}(x_{i,t}- x_{i,t}^*)-2C_2\eta \sum_{i=1}^N \|\mu_{i,t}-\bar{\lambda}_t\| \nonumber \\
			\geq & L_t(x_{t},\bar{\lambda}_t)-L_t(x_{t}^*,\bar{\lambda}_t)-2C_2\eta \sum_{i=1}^N \|\mu_{i,t}-\bar{\lambda}_t\| \label{lemma_Lt_first_eq3}.
		\end{align}
		Combining (\ref{lemma_Lt_first_eq2}) and (\ref{lemma_Lt_first_eq3}) yields (\ref{lemma_Lt_first_eq1}).
	\end{proof}
	
	\begin{lemma}\label{lemma_Lt_second}
		Suppose Assumptions \ref{AS4.1}-\ref{AS4.3} hold. For the sequences $\{x_{i,t}\}$ and $\{\lambda_{i,t}\}$ generated by  Algorithm \ref{algo2}, $t = 1, \dots, T$, $i=1,\dots, N$, and any $\lambda\in \mathbb{R}_+^n$, we have
		\begin{align}
			L_t(x_t, \lambda)-	L_t(x_t, \bar{\lambda}_t) \leq & \frac{1}{2N\alpha_t}\sum_{i=1}^N \left(\|\mu_{i,t}-\lambda\|^2 -\|\mu_{i,t+1}-\lambda\|^2 \right) + \frac{2C}{N} \sum_{i=1}^N \|\mu_{i,t}-\bar{\lambda}_t\| \nonumber \\
			&+ \frac{\alpha_t}{2N}\sum_{i=1}^N \left(2C^2+ 2\gamma_t^2\|\mu_{i,t}\|^2 \right).\label{lemma_Lt_second_eq1}
		\end{align}
	\end{lemma}
	\begin{proof}
		By Assumption \ref{AS4.1} and the property of the projection operator, we have 
		\begin{equation}\label{lemma_Lt_second_eq2}
			\|z_{i,t} -\gamma_t\mu_{i,t}\| \leq \|z_{i,t}\|+\gamma_t\|\mu_{i,t}\| \leq 2C,
		\end{equation}
		\begin{equation}\label{lemma_Lt_second_eq3}
			\|z_{i,t} -\gamma_t\mu_{i,t}\|^2 \leq 2\|z_{i,t}\|^2 +2\gamma_t^2\|\mu_{i,t}\| \leq 2C^2+ 2\gamma_t^2\|\mu_{i,t}\|^2 ,
		\end{equation}
		and 
		\begin{align}
		\sum_{i=1}^N  \|\mu_{i,t+1}-\lambda\|^2 =& 	\sum_{i=1}^N  \left \|\sum_{i=1}^N a_{ij}(t+1)\lambda_{j,t+1}-\lambda \right \|^2 \nonumber\\
		\leq & \sum_{i=1}^N \|\lambda_{i,t+1}-\lambda\|^2 \nonumber\\
		=& \sum_{i=1}^N \|\left[\mu_{i,t}+ \alpha_t\left(z_{i,t} -\gamma_t\mu_{i,t}\right) \right]_+-\lambda\|^2 \nonumber \\
		\leq & 	\sum_{i=1}^N  \left(\|\mu_{i,t}-\lambda\|^2 + \alpha_t^2\|z_{i,t} -\gamma_t\mu_{i,t}\|^2 +2\alpha_t \left(z_{i,t} -\gamma_t\mu_{i,t}\right)^{\top} (\mu_{i,t}-\lambda)\right).\label{lemma_Lt_second_eq4}
		\end{align}
		Since $L_t(x,\lambda)$ is concave with respect to $\lambda$, combining Assumptions \ref{AS4.2}-\ref{AS4.3} and (\ref{lemma_Lt_second_eq2}) yields
		\begin{align}
			&\sum_{i=1}^N \left(z_{i,t} -\gamma_t\mu_{i,t}\right)^{\top} (\mu_{i,t}-\lambda)\nonumber\\
			=& \sum_{i=1}^N \left(z_{i,t} -\gamma_t\mu_{i,t}\right)^{\top} (\bar{\lambda}_t-\lambda)+\sum_{i=1}^N \left(z_{i,t} -\gamma_t\mu_{i,t}\right)^{\top} (\mu_{i,t}-\bar{\lambda}_t) \nonumber\\
			=& N\left(\frac{1}{N} \sum_{i=1}^N z_{i,t}- \gamma_t \frac{1}{N} \sum_{i=1}^N\mu_{i,t} \right)^{\top}(\bar{\lambda}_t-\lambda)+\sum_{i=1}^N \left(z_{i,t} -\gamma_t\mu_{i,t}\right)^{\top} (\mu_{i,t}-\bar{\lambda}_t) \nonumber \\
			=& N\left( \sum_{i=1}^N g_{i,t}(x_{i,t})- \gamma_t \bar{\lambda}_t \right)^{\top}(\bar{\lambda}_t-\lambda)+\sum_{i=1}^N \left(z_{i,t} -\gamma_t\mu_{i,t}\right)^{\top} (\mu_{i,t}-\bar{\lambda}_t)\nonumber \\
			=& N \nabla_{\lambda} L_t(x_t, \bar{\lambda}_t)^{\top}(\bar{\lambda}_t-\lambda)+\sum_{i=1}^N \left(z_{i,t} -\gamma_t\mu_{i,t}\right)^{\top} (\mu_{i,t}-\bar{\lambda}_t)\nonumber \\
			\leq & N\left[	L_t(x_t, \bar{\lambda}_t)  -L_t(x_t, \lambda) \right] +2C\sum_{i=1}^N \|\mu_{i,t}-\bar{\lambda}_t\|. \label{lemma_Lt_second_eq5}
		\end{align}
		Combining (\ref{lemma_Lt_second_eq3}), (\ref{lemma_Lt_second_eq4}), and (\ref{lemma_Lt_second_eq5}) yields (\ref{lemma_Lt_second_eq1}).
	\end{proof}
	
	\begin{theorem}\label{4theorem_regret}
		Suppose Assumptions \ref{AS4.1}-\ref{AS4.3} hold. For the sequences $\{x_{i,t}\}$ and $\{\lambda_{i,t}\}$ generated by  Algorithm \ref{algo2}, $t = 1, \dots, T$, $i=1,\dots, N$, if $\alpha_t=t^{-\kappa_1}$, $\gamma_t=t^{-\kappa_2}$, $0<\kappa_1<\min\{2\kappa_2, 1-2\kappa_2\}$, $\kappa_2\in (0,\frac{1}{2})$, then we have
		\begin{align}
			Reg_T^D &= \mathcal{O}\left(T^{\max\{1-\kappa_1, 1-2\kappa_2+\kappa_1, 2\kappa_2+\kappa_1\}} +T^{\max\{\kappa_1, \kappa_2\}}P_T \right), \label{regret_order}\\
			Vio_T &= \mathcal{O}\left(T^{\max\left\{1-\frac{\kappa_2}{2}, \frac{1}{2}+\frac{\kappa_1}{2}\right\}} +T^{\max\{\frac{1}{2}, \frac{1}{2}-\frac{\kappa_2}{2}+ \frac{\kappa_1}{2}\}} \sqrt{P_T} \right)\label{violation_order}.
		\end{align}
	\end{theorem}
	\begin{proof}
		From (\ref{lemma_Lt_first_eq1}) and (\ref{lemma_Lt_second_eq1}), it follows that, for any $\lambda\in \mathbb{R}_+^n$, we have
		\begin{align}
			&L_t(x_t, \lambda)- L_t(x_t^*, \bar{\lambda}_t) \nonumber \\
			\leq & \frac{1}{2\alpha_t}\sum_{i=1}^N \! \left(\|x_{i,t}\!-\! x_{i,t}^*\|^2 \!-\! \|x_{i,t+1}\!-\! x_{i,t+1}^*\|^2 \right)\! +\! 2\!\left(\eta C_2+ \frac{C}{N}\right)\! \sum_{i=1}^N \|\mu_{i,t}-\bar{\lambda}_t\| \nonumber \\
			&+ \sum_{i=1}^N \left(\frac{3\eta}{\alpha_t}+ C_2\left(1+ \|\mu_{i,t}\|  \right) \right)\|x_{i,t}^* - x_{i,t+1}^*\|+\alpha_t C_2^2\sum_{i=1}^N\left(1+ \|\mu_{i,t}\|^2\right) \nonumber\\
			& +\frac{1}{2N\alpha_t}\sum_{i=1}^N \left(\|\mu_{i,t}-\lambda\|^2 -\|\mu_{i,t+1}-\lambda\|^2 \right) + \frac{\alpha_t}{2N}\sum_{i=1}^N \left(2C^2+ 2\gamma_t^2\|\mu_{i,t}\|^2 \right)\label{theorem_regret_q1}.
		\end{align}
		By the definition of $L_t$ and $\bar{\lambda}_t^{\top} \sum_{i=1}^N g_{i,t}(x_{i,t}^*) \leq 0$, we have 
		\begin{align}
			& L_t(x_t, \lambda)- L_t(x_t^*, \bar{\lambda}_t) \nonumber\\
			=& \sum_{i=1}^N \left(f_{i,t}(x_{i,t}) - f_{i,t}(x_{i,t}^*) \right) +\lambda^{\top}\sum_{i=1}^N g_{i,t}(x_{i,t}) -\frac{\gamma_t}{2}\|\lambda\|^2 -\bar{\lambda}_t^{\top}\sum_{i=1}^Ng_{i,t}(x_{i,t}^*) +\frac{\gamma_t}{2}\|\bar{\lambda}_t\|^2 \nonumber \\
			\geq & \sum_{i=1}^N \left(f_{i,t}(x_{i,t}) - f_{i,t}(x_{i,t}^*) \right) +\lambda^{\top}\sum_{i=1}^N g_{i,t}(x_{i,t}) -\frac{\gamma_t}{2}\|\lambda\|^2  +\frac{\gamma_t}{2}\|\bar{\lambda}_t\|^2. \label{theorem_regret_q2}
		\end{align}
		Combining (\ref{theorem_regret_q1}) and (\ref{theorem_regret_q2}), we have
		\begin{align}
			&Reg_T^D  + \lambda^{\top}\sum_{t=1}^{T}\sum_{i=1}^N g_{i,t}(x_{i,t}) -\sum_{t=1}^{T}\frac{\gamma_t}{2}\|\lambda\|^2  +\sum_{t=1}^{T}\frac{\gamma_t}{2}\|\bar{\lambda}_t\|^2 \nonumber\\
			\leq & \sum_{t=1}^{T}\frac{1}{2\alpha_t}\sum_{i=1}^N \! \left(\|x_{i,t}\!-\! x_{i,t}^*\|^2 \!-\! \|x_{i,t+1}\!-\! x_{i,t+1}^*\|^2 \right)\! +\! 2\!\left(\!\!\eta C_2+\!\! \frac{C}{N}\right)\! \sum_{t=1}^{T}\sum_{i=1}^N \|\mu_{i,t}-\bar{\lambda}_t\| \nonumber \\
			&+ \sum_{t=1}^{T}\sum_{i=1}^N \left(\frac{3\eta}{\alpha_t}+ C_2\left(1+ \|\mu_{i,t}\|  \right) \right)\|x_{i,t}^* - x_{i,t+1}^*\|+\sum_{t=1}^{T}\alpha_t C_2^2\sum_{i=1}^N\left(1+ \|\mu_{i,t}\|^2\right) \nonumber\\
			& +\sum_{t=1}^{T}\frac{1}{2N\alpha_t}\sum_{i=1}^N \left(\|\mu_{i,t}-\lambda\|^2 -\|\mu_{i,t+1}-\lambda\|^2 \right) + \sum_{t=1}^{T}\frac{\alpha_t}{2N}\sum_{i=1}^N \left(2C^2+ 2\gamma_t^2\|\mu_{i,t}\|^2 \right)\label{theorem_regret_q3}.
		\end{align}
         Next, we analyze each term in (\ref{theorem_regret_q3}). First, we have
		\begin{equation}\label{theorem_regret_q4}
			\sum_{t=1}^{T} \alpha_t =\sum_{t=2}^{\top} \frac{1}{t^{\kappa_1}} +1 \leq \int_1^{\top} \frac{1}{t^{\kappa_1}}dt +1 \leq \frac{T^{1-\kappa_1}}{1-\kappa_1}.
		\end{equation}
		By the above equation and Lemma \ref{lam_consensus}, we have
		\begin{equation}\label{theorem_regret_q5}
			2 \left(\eta C_2+\frac{C}{N}\right) \sum_{t=1}^{T}\sum_{i=1}^N \|\mu_{i,t}-\bar{\lambda}_t\| \leq 2D\left(\eta C_2+\frac{C}{N}\right)\frac{T^{1-\kappa_1}}{1-\kappa_1}.
		\end{equation}
		From Lemma \ref{lambda_bounded}, it follows that 
		\begin{equation}\label{theorem_regret_q6}
		\sum_{t=1}^{T}\sum_{i=1}^N \left(\frac{3\eta}{\alpha_t}+ C_2\left(1+ \|\mu_{i,t}\|  \right) \right)\|x_{i,t}^* - x_{i,t+1}^*\| \leq \left[3\eta T^{\kappa_1} +C_2\left(1+CT^{\kappa_2} \right) \right]P_T.
		\end{equation}
		Combining (\ref{theorem_regret_q4}), Lemma \ref{lambda_bounded}, and $0<2\kappa_2-\kappa_1<1$ yields
		\begin{align}
			&\sum_{t=1}^{T}\alpha_t C_2^2\sum_{i=1}^N\left(1+ \|\mu_{i,t}\|^2\right) + \sum_{t=1}^{T}\frac{\alpha_t}{2N}\sum_{i=1}^N \left(2C^2+ 2\gamma_t^2\|\mu_{i,t}\|^2 \right) \nonumber\\
			\leq & \left(C_2^2 + \frac{2C^2}{N} \right)\frac{T^{1-\kappa_1}}{1-\kappa_1} + C_2^2 C^2 \sum_{t=1}^{T}  t^{2\kappa_2 -\kappa_1}\nonumber \\
			\leq &\left(C_2^2 + \frac{2C^2}{N} \right)\frac{T^{1-\kappa_1}}{1-\kappa_1} +\frac{C_2^2 C^2 T^{1-2\kappa_2+\kappa_1}}{1-2\kappa_2+\kappa_1}      .\label{theorem_regret_q7}
		\end{align}
		By Assumption \ref{AS4.2}, we have
		\begin{align}
		&\sum_{t=1}^{T}\frac{1}{2\alpha_t}\sum_{i=1}^N \! \left(\|x_{i,t}\!-\! x_{i,t}^*\|^2 \!-\! \|x_{i,t+1}\!-\! x_{i,t+1}^*\|^2 \right) \nonumber \\
		=& \sum_{i=1}^N \left[\frac{1}{2\alpha_1}\|x_{i,1} - x_{i,2}^*\|^2  -  \frac{1}{2\alpha_T}\|x_{i,T+1} - x_{i,T+1}^*\|^2 +\sum_{t=1}^{T-1}\!\left(\frac{1}{2\alpha_{t+1}}-\frac{1}{2\alpha_{t}} \right)\! \|x_{i,t+1}-x_{i,t+1}^*\|^2 \right] \nonumber \\
		\leq & 2N\eta^2 +2N\eta^2\sum_{t=1}^{T-1} \left(\frac{1}{2\alpha_{t+1}}-\frac{1}{2\alpha_{t}} \right) \nonumber\\
		\leq & 2N\eta^2T^{\kappa_1}.\label{theorem_regret_q8}
		\end{align}
		Similarly, by $\mu_{i,1}=0$, we have
		\begin{align}
			&\sum_{t=1}^{T}\frac{1}{2N\alpha_t}\sum_{i=1}^N \left(\|\mu_{i,t}-\lambda\|^2 -\|\mu_{i,t+1}-\lambda\|^2 \right) \nonumber\\
			\leq & \frac{1}{2N} \sum_{i=1}^N \left[\|\lambda\|^2+ \sum_{t=1}^{t-1} \left(\frac{1}{\alpha_{t+1}}-\frac{1}{\alpha_{t}} \right)\|\mu_{i,t+1}- \lambda \|^2 \right]  \nonumber \\
			\leq & \frac{T^{\kappa_1}}{2}\|\lambda\|^2+ \frac{1}{N}\sum_{t=1}^{T}\sum_{i=1}^N \left(\frac{1}{\alpha_{t+1}}-\frac{1}{\alpha_{t}} \right)\frac{C^2}{\gamma_t^2} \nonumber \\
			\leq & \frac{T^{\kappa_1}}{2}\|\lambda\|^2+ {C^2}{T^{2\kappa_2+\kappa_1}}.\label{theorem_regret_q9}
		\end{align}
		Substituting (\ref{theorem_regret_q5})–(\ref{theorem_regret_q9}) into (\ref{theorem_regret_q3}) yields
		\begin{align}
			&Reg_T^D  + \lambda^{\top}\sum_{t=1}^{T}\sum_{i=1}^N g_{i,t}(x_{i,t}) -\sum_{t=1}^{T}\frac{\gamma_t}{2}\|\lambda\|^2- \frac{T^{\kappa_1}}{2}\|\lambda\|^2 +\sum_{t=1}^{T}\frac{\gamma_t}{2}\|\bar{\lambda}_t\|^2 \nonumber\\
			\leq & \left[2D\left(\eta C_2+\frac{C}{N}\right)+C_2^2 + \frac{2C^2}{N} \right]\frac{T^{1-\kappa_1}}{1-\kappa_1} + \left[3\eta T^{\kappa_1} +C_2\left(1+CT^{\kappa_2} \right) \right]P_T+ 2N\eta^2T^{\kappa_1} \nonumber\\
			& +\frac{C_2^2 C^2 T^{1-2\kappa_2+\kappa_1}}{1-2\kappa_2+\kappa_1}   +{C^2}{T^{2\kappa_2+\kappa_1}}.\label{theorem_regret_q10}
		\end{align}
		Setting $\lambda=0$ in (\ref{theorem_regret_q10}), we obtain
		\begin{align*}
			Reg_T^D \leq&  \left[2D\left(\eta C_2+\frac{C}{N}\right)\!+C_2^2 + \frac{2C^2}{N} \right]\frac{T^{1-\kappa_1}}{1-\kappa_1} + \left[3\eta T^{\kappa_1} +C_2\left(1+CT^{\kappa_2} \right) \right]P_T+ 2N\eta^2T^{\kappa_1} \nonumber \\
			& +\frac{C_2^2 C^2 T^{1-2\kappa_2+\kappa_1}}{1-2\kappa_2+\kappa_1}   +{C^2}{T^{2\kappa_2+\kappa_1}}\nonumber .
		\end{align*}
		Hence, we obtain (\ref{regret_order}).
		
		Let $F(\lambda)= \lambda^{\top}\sum_{t=1}^{T}\sum_{i=1}^N g_{i,t}(x_{i,t}) -\sum_{t=1}^{T}\frac{\gamma_t}{2}\|\lambda\|^2- \frac{T^{\kappa_1}}{2}\|\lambda\|^2$. By (\ref{theorem_regret_q10}), we have
		\begin{align*}
			Reg_T^D + F(\lambda) \leq&  \left[2D\left(\eta C_2+\frac{C}{N}\right)\!+C_2^2 + \frac{2C^2}{N} \right]\frac{T^{1-\kappa_1}}{1-\kappa_1} + \left[3\eta T^{\kappa_1} +C_2\left(1+CT^{\kappa_2} \right) \right]P_T\nonumber \\
			& + 2N\eta^2T^{\kappa_1} +\frac{C_2^2 C^2 T^{1-2\kappa_2+\kappa_1}}{1-2\kappa_2+\kappa_1}   +{C^2}{T^{2\kappa_2+\kappa_1}}\nonumber .
		\end{align*}
		Since $F(\lambda)$ is concave with respect to $\lambda$, it has a maximum. Computing this maximum and substituting it into the above equation gives
			\begin{align}
			&Reg_T^D + \left[2\left(\sum_{t=1}^{T} \gamma_t +T^{\kappa_1} \right) \right]^{-1} \left\|\left[\sum_{t=1}^{T} \sum_{i=1}^N g_{i,t}(x_{i,t}) \right]_+ \right\|^2 \nonumber \\
			\leq&  \left[2D\left(\eta C_2+\frac{C}{N}\right)\!+C_2^2 + \frac{2C^2}{N} \right]\frac{T^{1-\kappa_1}}{1-\kappa_1} + \left[3\eta T^{\kappa_1} +C_2\left(1+CT^{\kappa_2} \right) \right]P_T\nonumber \\
			& + 2N\eta^2T^{\kappa_1} +\frac{C_2^2 C^2 T^{1-2\kappa_2+\kappa_1}}{1-2\kappa_2+\kappa_1}   +{C^2}{T^{2\kappa_2+\kappa_1}}\label{theorem_regret_q11} .
		\end{align}
		By Assumptions \ref{AS4.2}-\ref{AS4.3}, we have 
		$$Reg_T^D  \geq -\sum_{t=1}^{T} \sum_{i=1}^N C_2\|x_{i,t}-x_{i,t}^*\| \geq -2C_2\eta NT. $$
		Substituting the above equation into (\ref{theorem_regret_q11}) yields
		\begin{align*}
			 & \left\|\left[\sum_{t=1}^{T} \sum_{i=1}^N g_{i,t}(x_{i,t}) \right]_+ \right\|^2 \\
			 \leq & \bigg[ \left[2D\left(\eta C_2+\frac{C}{N}\right)\!+C_2^2 + \frac{2C^2}{N} \right]\frac{T^{1-\kappa_1}}{1-\kappa_1} + \left[3\eta T^{\kappa_1} +C_2\left(1+CT^{\kappa_2} \right) \right]P_T+ 2N\eta^2T^{\kappa_1} \\
			 &+\frac{C_2^2 C^2 T^{1-2\kappa_2+\kappa_1}}{1-2\kappa_2+\kappa_1}  +{C^2}{T^{2\kappa_2+\kappa_1}} +2C_2\eta NT\bigg] \left( \frac{2T^{1-\kappa_2}}{1-\kappa_2} + 2T^{\kappa_1} \right).
		\end{align*}
		Therefore, (\ref{violation_order}) holds.

	\end{proof}

\begin{remark}
	If $\alpha_t = \gamma_t = t^{-\frac{1}{4}}$, then the step size and regularization coefficient satisfy the conditions in Theorem \ref{4theorem_regret}, and we have $Reg_T^D =  \mathcal{O}\left(T^{\frac{3}{4}} +T^{\frac{1}{4}}P_T\right)$ and $Vio_T=\mathcal{O}\left(T^{\frac{7}{8}} +T^{\frac{1}{2}}\sqrt{P_T}\right) $. If the path length of the optimal solutions satisfies $P_T \leq \mathcal{O}(T^{\frac{3}{4}})$, then the dynamic regret and the constraint violation of Algorithm \ref{algo2} are both sublinear, i.e., $\lim_{T\rightarrow \infty} Reg_T^D/T=0$ and $\lim_{T\rightarrow \infty} Vio_T/T=0$.
\end{remark}

\section{Numerical experiments}

To assess the proposed distributed online algorithm for online economic dispatch, we carry out two numerical experiments: Section~\ref{sec:sim} considers synthetic datasets, while Section~\ref{sec:elec} focuses on real-world datasets involving the electricity price and demand data from five Australia states, i.e., NSW, QLD, SA, VIC, TAS. All experiments are implemented in MATLAB R2023b and executed on a computer equipped with AMD Ryzen 7 processor (1.80 GHz) and 16 GB of RAM.

\subsection{Synthetic data}
\label{sec:sim}

In this subsection, we consider an $N$-bus microgrid system following the setup in \cite{wangcong2024}. The microgrid is interconnected with neighboring microgrids as well as the main grid via energy routers. Each node contains a distributed generator and a local load, and each generator is equipped with an intelligent control unit that acts as a local controller. At each time $t$, generator $i$ is characterized by a time-varying cost function and time-varying constraint functions. The agents cooperate to optimize the following problem
\begin{equation}\nonumber
	\begin{aligned}
		\min _{x_i\in \Omega_i } &\sum_{i=1}^N a_{i,t} x_i^{\top} x_i + b_{i,t}^{\top}x_i,\\
		\text{s.t.}\,\, &  \sum_{i=1}^N c_{i,t} x_i^{\top} x_i + q_{i,t}^{\top} x_{i,t} + e_{i,t} \leq 0, \, t=1,\dots T,
	\end{aligned}
\end{equation}
where $x_i$ is the power output of generator $i$, $a_{i,t} x_i^{\top} x_i + b_{i,t}^{\top}x_i$ is the cost function of generator $i$ at time $t$, $\sum_{i=1}^N c_{i,t} x_i^{\top} x_i + q_{i,t}^{\top} x_{i,t} + e_{i,t} \leq 0$ is the global power-balance constraint. The initial coefficients $b_{i,1}$ and $q_{i,1}$ are drawn uniformly from $[0,10]^{d_i}$ and $[0,5]^{d_i}$, respectively; $a_{i,1}$, $c_{i,1}$, and $e_{i,1}$ are sampled uniformly from $[1,15]$, $[1,5]$, and $[-30,0]$, respectively. For $t>1$, the coefficients are given as $b_{i,t}=b_{i,t-1}+\epsilon_1$, $q_{i,t}=q_{i,t-1}+\epsilon_1$, $a_{i,t}=a_{i,t-1}+\epsilon_2$, $c_{i,t}=c_{i,t-1}+\epsilon_2$, $e_{i,t}=e_{i,t-1}+10\epsilon_2$, where $\epsilon_1$ is drawn uniformly from $[-0.05, 0.05]^{d_i}$ and $\epsilon_2$ uniformly from $[-0.05, 0.05]$. These time-varying coefficients capture fluctuation in prices and demand in the power market. Each generator's feasible set is given by $\Omega_i=\{ x_i | \|x_i\| \leq 300\}$. To model realistic distributed coordination, the communication network among generators is represented by three switching graphs, as shown in Figure \ref{figgraph}.
The resulting communication graph satisfies Assumption \ref{AS4.1} ($B=3$), ensuring that information can propagate across all nodes.

\begin{figure}[htbp]
	\centering	
	\subfigure %第一张子图
	{
		\centering          %子图居中
		\includegraphics[scale=0.34]{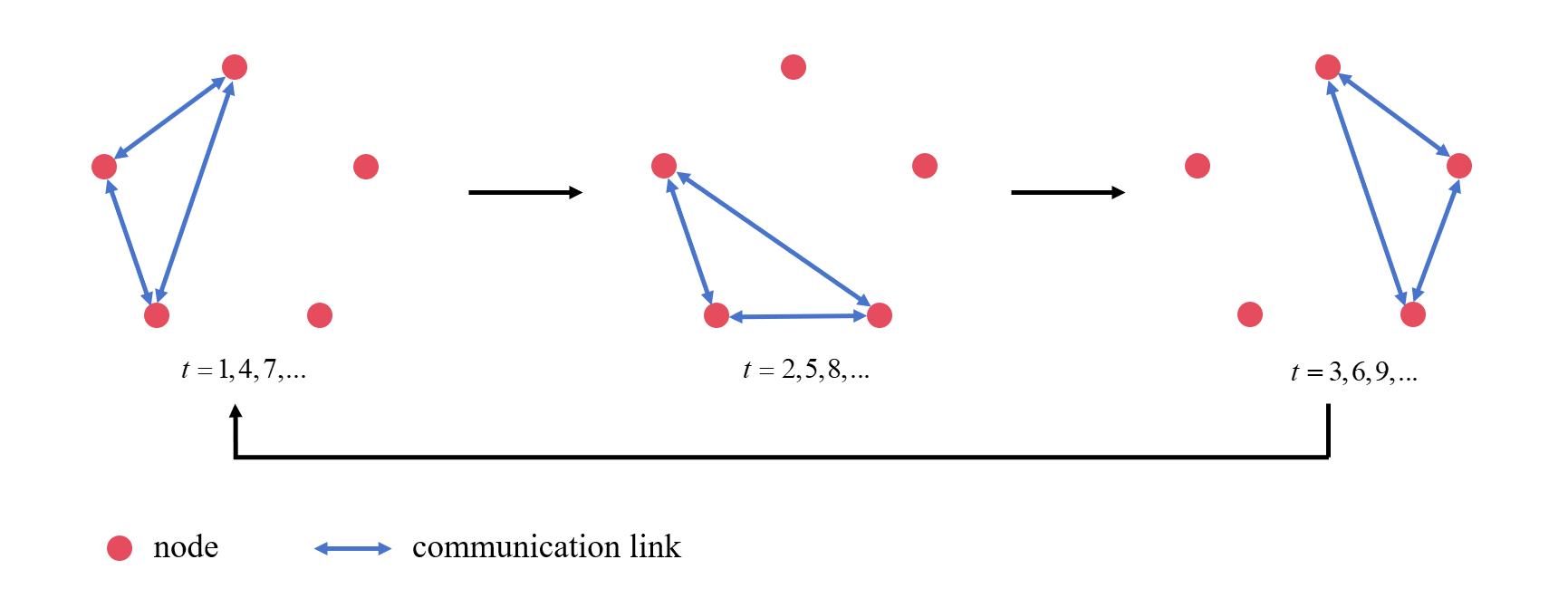}
		%	\caption{XX}
	}
	\caption{The communication network. } %  % 大图名称
	\label{figgraph}
\end{figure}

We first examine the effect of step sizes on the performance of the proposed algorithm. We consider three schemes: (1) $\alpha_t= \gamma_t= t^{-\frac{1}{4}}$, (2) $\alpha_t= \gamma_t = t^{-\frac{1}{8}}$, and (3) $\alpha_t = t^{-\frac{1}{3}}, \gamma_t=t^{-\frac{1}{4}}$. The variable dimension is fixed at $d_i=3$ for all $i=1,\dots, N$. We set $N = 5$ generators and run the algorithm for $T = 200$ time periods.
Figure \ref{fig1} shows the resulting dynamic regret and constraint violation under these different step sizes. All three choices lead to sublinear growth in both metrics, given that the step sizes satisfy the conditions established in Theorem \ref{4theorem_regret}. In particular, because Algorithm \ref{algo2} uses a gradient-tracking-like mechanism to track global coupled constraint functions, constraint violations diminish rapidly and approach zero. This behavior indicates that the proposed algorithm adapts efficiently to the time-varying coupled constraint and that subsequent iterates consistently satisfy the inequality constraints.

\begin{figure}[htbp]
	\centering	
	\subfigure[] %第一张子图
	{
		\centering          %子图居中
		\includegraphics[scale=0.5]{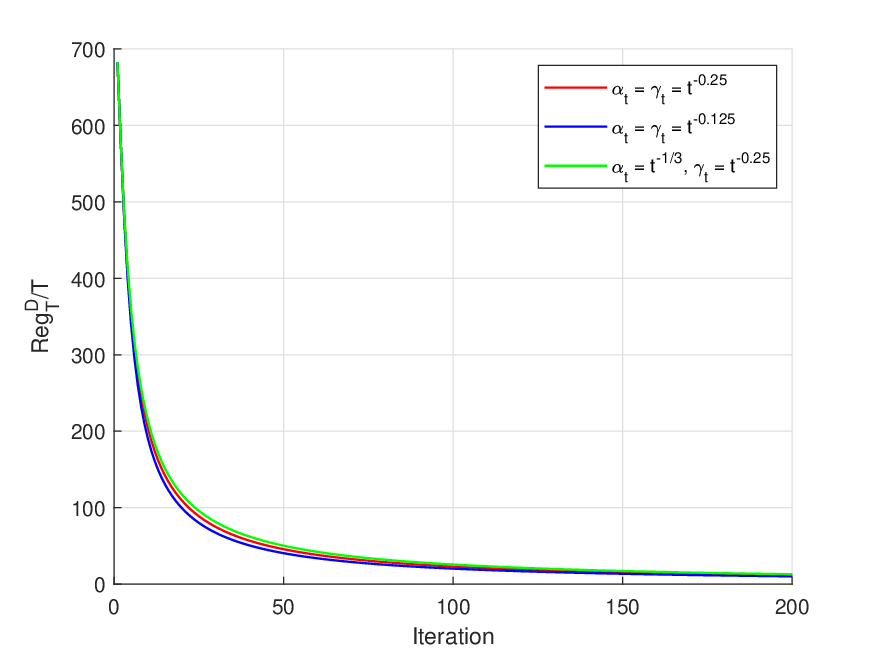}
		%	\caption{XX}
	}
	\subfigure[] %第一张子图
	{
		\centering          %子图居中
		\includegraphics[scale=0.5]{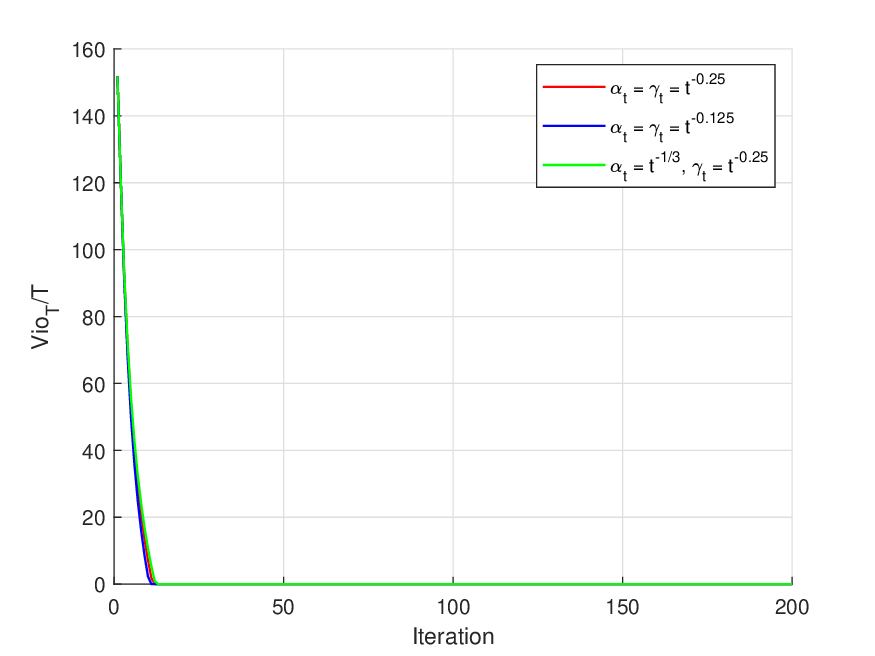}
		%	\caption{XX}
	}
	\caption{Results under different step sizes: (a) dynamic regret, (b) constraint violation.  } %  %大图名称
	\label{fig1}
\end{figure}

Then, we compare the proposed Algorithm \ref{algo2} with the benchmark algorithm in \cite{Lijueyou2022}. The benchmark algorithm uses the best step size recommended in their paper, while Algorithm \ref{algo2} uses $\alpha_t = \gamma_t = t^{-\frac{1}{4}}$. As shown in Figure \ref{fig3}, although both algorithms achieve sublinear dynamic regret and constraint violation, Algorithm \ref{algo2} provides superior performance in terms of both dynamic regret and constraint violation. Moreover, the benchmark algorithm in \cite{Lijueyou2022} relies on prior knowledge (specifically, a global upper bound on the gradients of all nodes' objective functions over time), whereas our algorithm requires no such assumption, and is therefore more applicable in a wider range of settings.

\begin{figure}[htbp]
	\centering	
	\subfigure[] %第一张子图
	{
		\centering          %子图居中
		\includegraphics[scale=0.5]{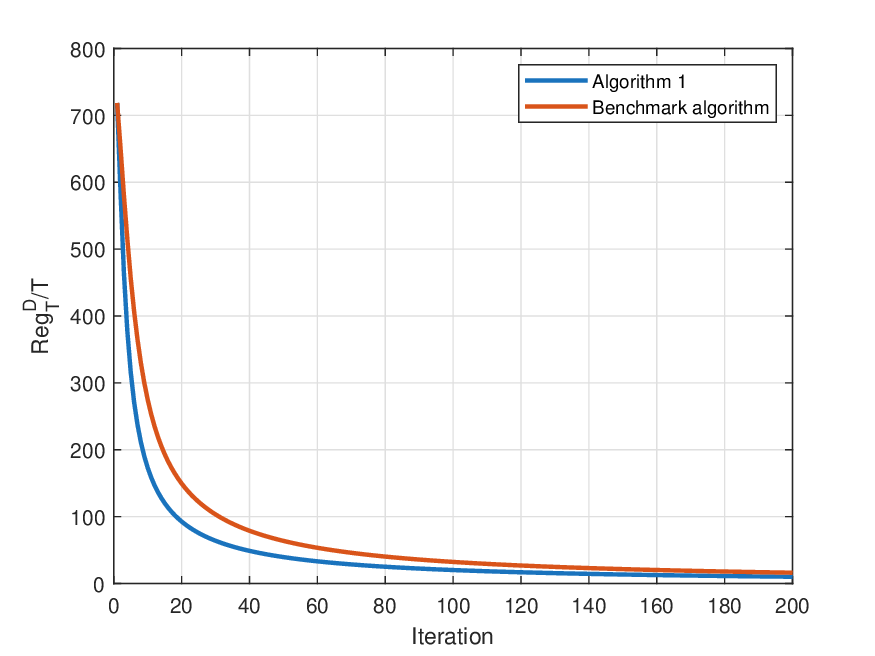}
		%	\caption{XX}
	}
	\subfigure[] %第一张子图
	{
		\centering          %子图居中
		\includegraphics[scale=0.5]{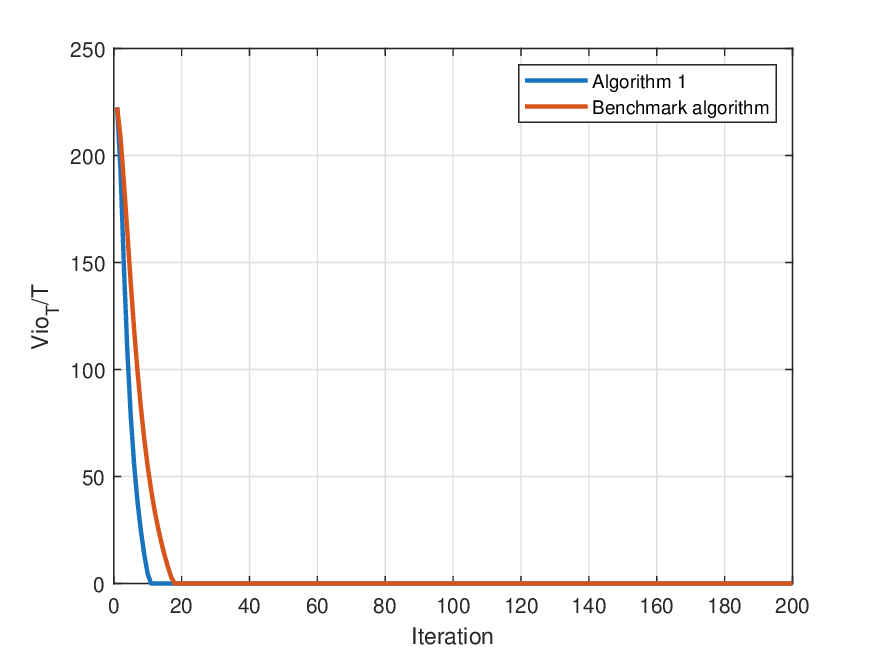}
		%	\caption{XX}
	}
	\caption{Experimental results comparing with the benchmark algorithm in \cite{Lijueyou2022}: (a) dynamic regret, (b) constraint violation.  } %  %大图名称
	\label{fig3}
\end{figure}

\subsection{Australian electricity data}
\label{sec:elec}

In this subsection, we study the distributed online economic dispatch problem formulated in (\ref{online_dispatch}), which considers the objective of minimizing generation costs while satisfying time-varying demand constraints. The experimental data are obtained from publicly available real-time electricity price (\$/MWh) and demand (MW) records provided by the Australian Energy Market Operator (AEMO) for five states, including NSW, QLD, SA, VIC, and TAS (\url{https://www.aemo.com.au/}). Each dataset spans the period from 00:00 on July 5, 2024 to 23:55 on July 14, 2024, with observations recorded at $5$-minute intervals, resulting in a total of $2880$ data points ($T=2880$). The corresponding time series for electricity prices and demand across the five states are shown in Figure \ref{time_series}. In this setting, $P_t$ and $D_t$ denote the time-varying electricity price and demand, respectively.

In practical economic dispatch problems, operational decisions are typically made following a forecast–then–optimize (PTO) paradigm, in which electricity prices and demand are first forecast based on historical information \citep[e.g.,][]{Panagiotelis2008,Fan2011,Jeon2019} and the resulting forecasts are then used to determine generator outputs by solving an optimization problem that minimizes operating costs while ensuring supply–demand balance. In this paper, our focus is on the decision-making stage given forecasted inputs, rather than on the forecasting task itself. To isolate the performance of the proposed distributed online optimization algorithm and avoid confounding effects from forecasting errors, we assume that the electricity price and demand available at each time $t$ are provided as ``ideal" one-step-ahead forecasts generated from historical data up to time $t-1$. Accordingly, these values are treated as the forecast inputs to the dispatch problem, and all subsequent numerical results are based on decisions computed using these values.

\begin{figure}[htbp]
  \centering
  \includegraphics[scale=0.43]{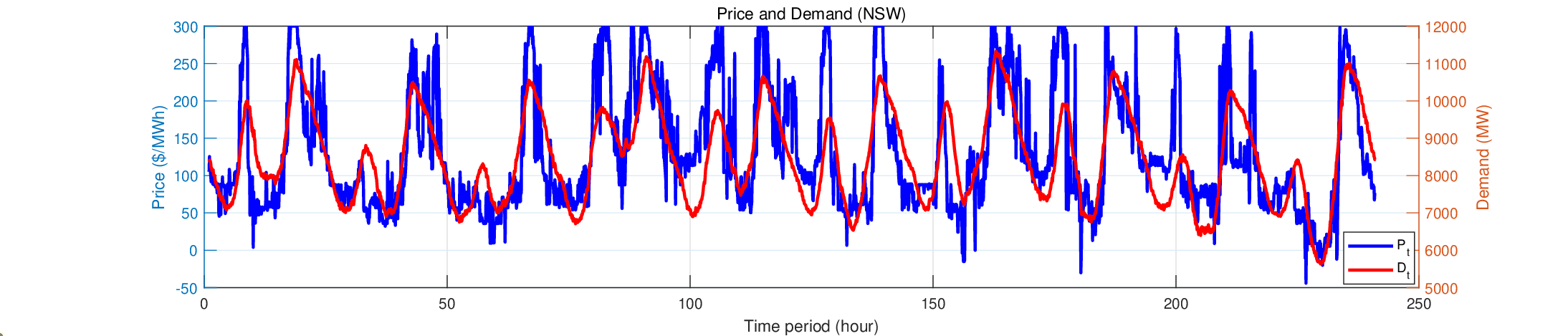}\par\vspace{2mm}
  \includegraphics[scale=0.43]{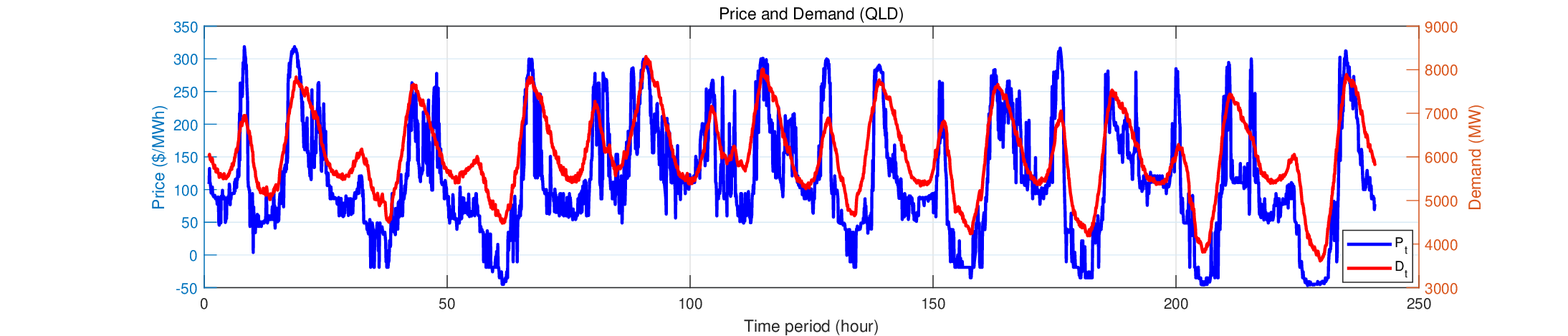}\par\vspace{2mm}
  \includegraphics[scale=0.43]{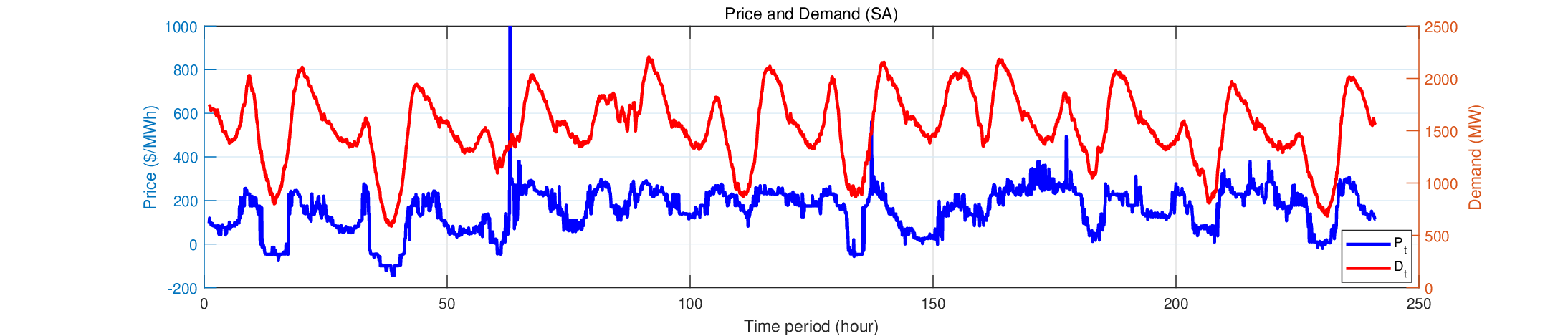}\par\vspace{2mm}
  \includegraphics[scale=0.43]{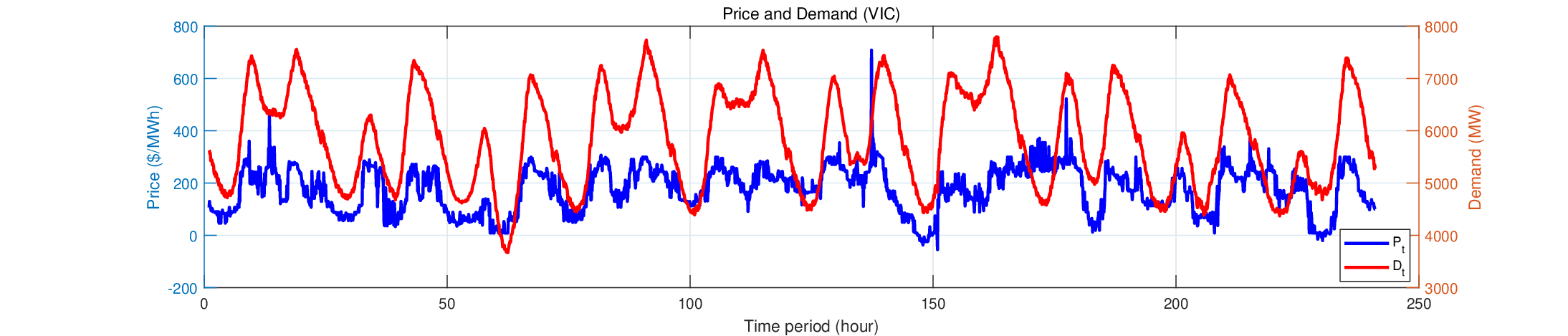}\par\vspace{2mm}
  \includegraphics[scale=0.43]{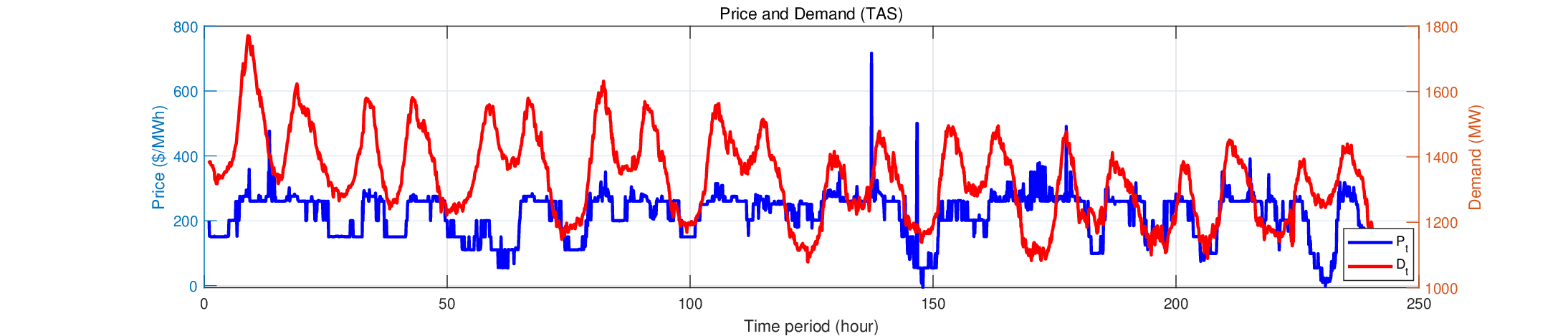}
   \caption{The original data, from top to bottom: NSW, QLD, SA, VIC, TAS.}
  \label{time_series}
\end{figure}

In this experiment, we consider a system with $5$ generators. Due to the difficulty in obtaining generator cost data at every time instant, we adopt time-invariant cost coefficients, i.e., $a_{i,t}=a_i, b_{i,t}= b_i, c_{i,t}=c_i$, $t=1, \dots, T$, $i=1, \dots, N$. Each generator is modeled with a quadratic cost function, with coefficients $a=[0.040; 0.050; 0.035; 0.045; 0.038]$, $b=[-0.1200; -0.1500;$ $ -0.1050; -0.1350; -0.1140]$, $c=[100; 110; 95; 105; 98]$. The communication topology is also given in Figure \ref{figgraph}.

We compare the performance of Algorithm \ref{algo2} with the benchmark algorithm in \cite{Lijueyou2022}.
The benchmark algorithm relies on a fixed step size that incorporates prior knowledge, which is often hard to obtain in advance. In contrast, our proposed algorithm does not depend on such prior information. As long as the algorithm's step size satisfies the conditions stated in Theorem \ref{4theorem_regret}, our algorithm is feasible. Furthermore, as shown in Figure \ref{fig1}, the performance of the proposed algorithm is insensitive to the choice of step sizes. Therefore, our Algorithm \ref{algo2} is more convenient to use in practice.

We set the step size as $\alpha_t= \gamma_t= t^{-\frac{1}{4}}$. Table \ref{tab:comp_time} reports the computation time for each state over the full horizon $T=2880$. The results show that the proposed algorithm remains computationally efficient even for relatively large datasets, as it processes data in an online manner and updates decisions using only information from the current and immediately preceding time steps, without storing the full history.

Figures \ref{figNSW}-\ref{figTAS} present the resulting dynamic regret and constraint violation over time for the five states. The results indicate that both metrics grow sublinearly, in the sense that $Reg_T^D /T$ and $Vio_T /T$ converge to zero as $T \rightarrow \infty$. This indicates that the algorithm is able to effectively track the time-varying optimal solution while maintaining feasibility on average, highlighting its practicality and usefulness for distributed economic dispatch. Moreover, compared with the benchmark algorithm, Algorithm \ref{algo2} achieves faster convergence in both dynamic regret and constraint violation, underscoring its practical advantages.

As shown in Figures \ref{figNSW}-\ref{figTAS}, both the dynamic regret and the constraint violation exhibit noticeable fluctuations at the initial stage. This behavior arises because the initial decision variables are chosen randomly and are not close to the optimal solutions, preventing the algorithm from immediately tracking the optimum and requiring an initial adaptation period. In addition, numerous smaller fluctuations in dynamic regret and constraint violation can be observed throughout the time horizon. These fluctuations are primarily driven by demand peaks and troughs that induce abrupt changes in electricity demand and prices, thereby increasing tracking errors and temporarily elevating both regret and constraint violations.

\begin{table}[htbp]
\centering
\caption{Computation time of the proposed algorithm for each state.}
\label{tab:comp_time}
\begin{tabular}{lccccc}
\toprule
State & NSW & QLD & SA & VIC & TAS \\
\midrule
Computation time (s) & 45.54  & 38.85 & 43.22 & 39.01  & 39.85 \\
\bottomrule
\end{tabular}
\end{table}

\begin{figure}[htbp]
	\centering	
	\subfigure[] %第一张子图
	{
		\centering          %子图居中
		\includegraphics[scale=0.5]{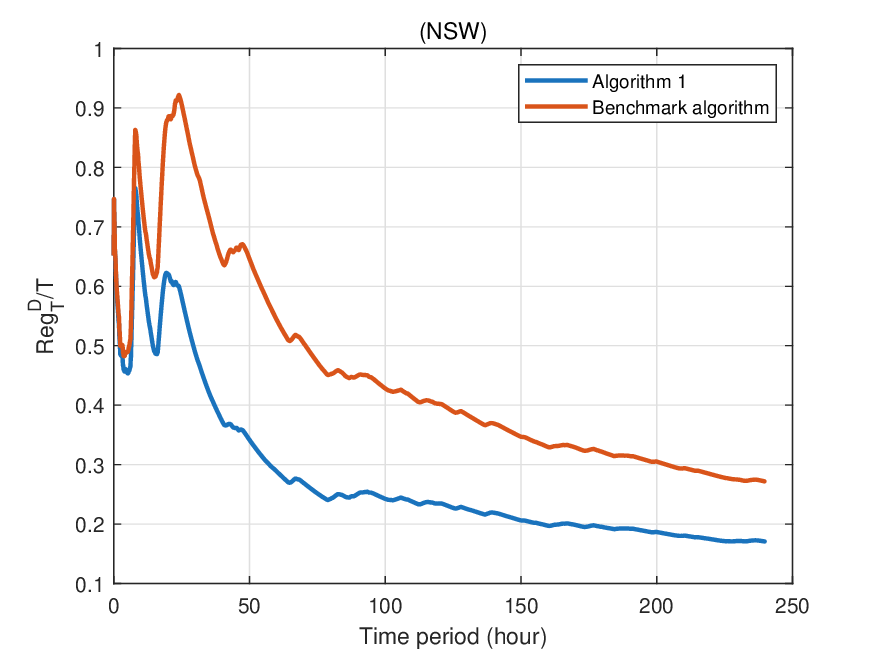}
		%	\caption{XX}
	}
    \subfigure[] %第一张子图
	{
		\centering          %子图居中
		\includegraphics[scale=0.5]{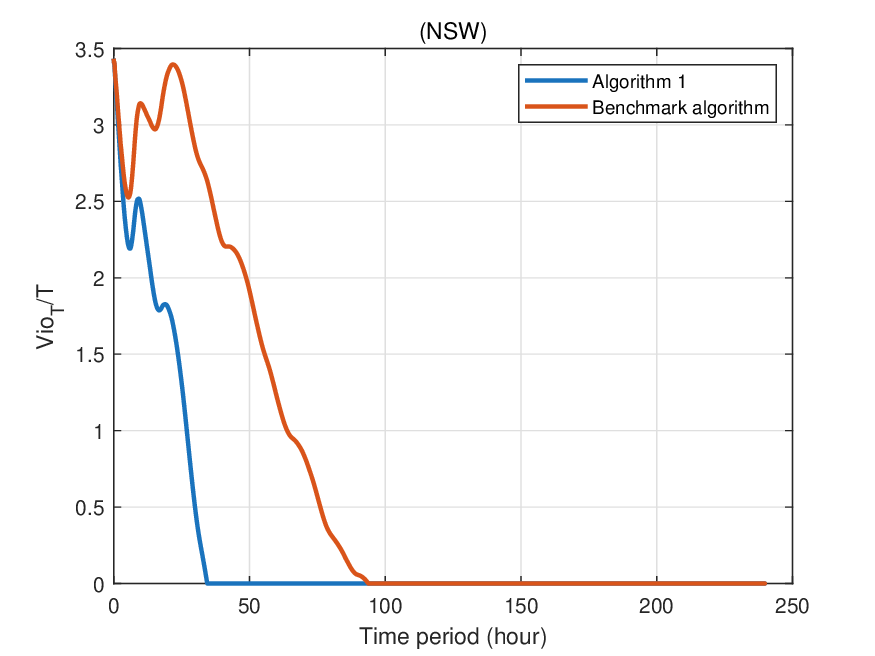}
		%	\caption{XX}
	}
	\caption{Experiment results from NSW: (a) dynamic regret, (b) constraint violation. } 
	\label{figNSW}
\end{figure}

\begin{figure}[htbp]
	\centering	
	\subfigure[] %第一张子图
	{
		\centering          %子图居中
		\includegraphics[scale=0.5]{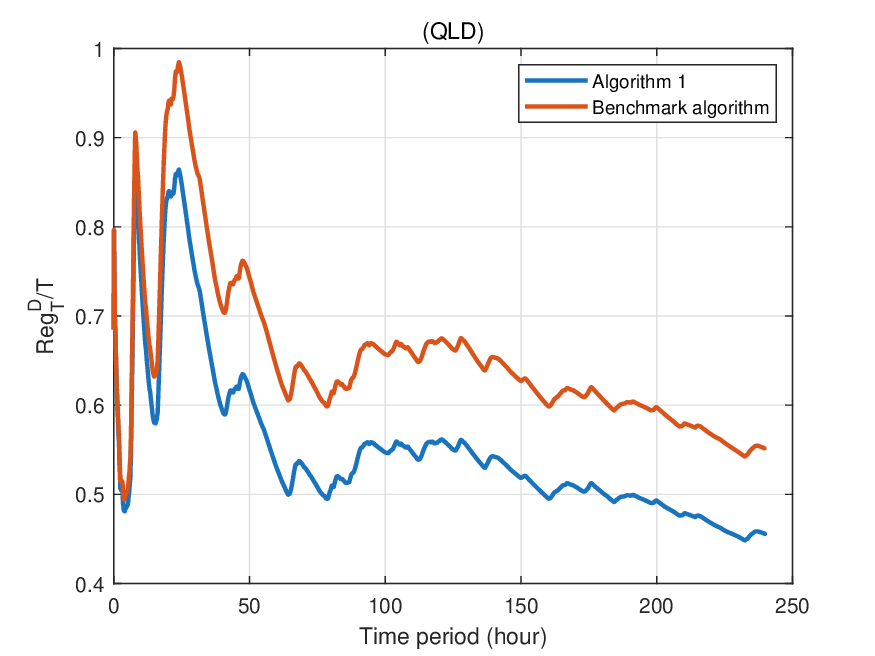}
		%	\caption{XX}
	}
    \subfigure[] %第一张子图
	{
		\centering          %子图居中
		\includegraphics[scale=0.5]{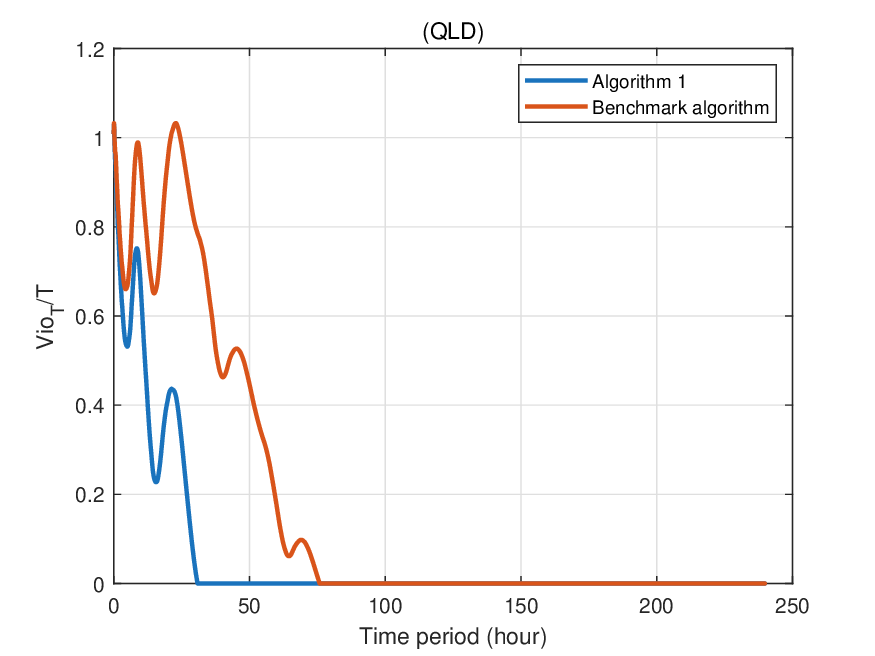}
		%	\caption{XX}
	}
	\caption{Experiment results from QLD: (a) dynamic regret, (b) constraint violation. } 
	\label{figQLD}
\end{figure}

\begin{figure}[htbp]
	\centering	
	\subfigure[] %第一张子图
	{
		\centering          %子图居中
		\includegraphics[scale=0.5]{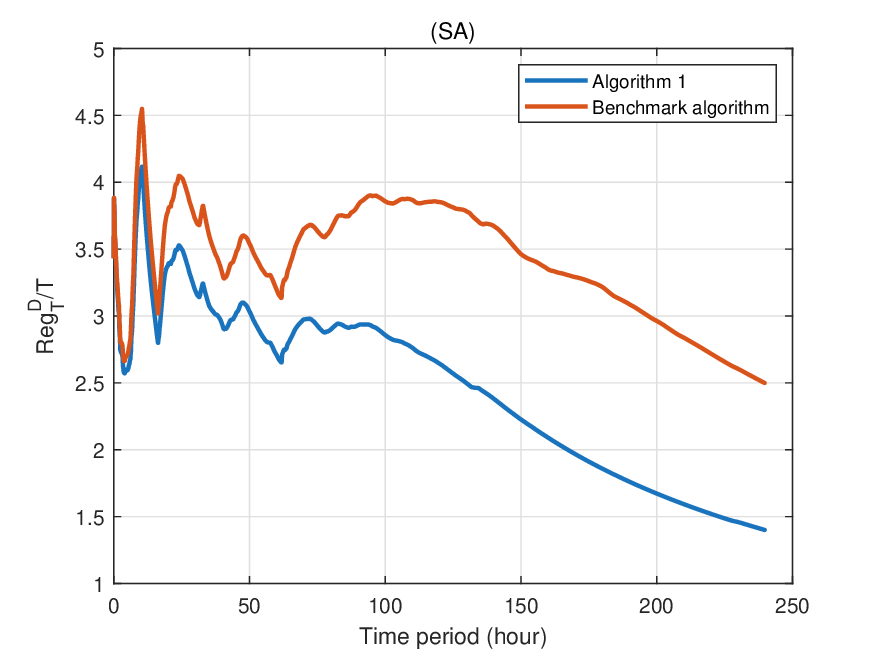}
		%	\caption{XX}
	}
    \subfigure[] %第一张子图
	{
		\centering          %子图居中
		\includegraphics[scale=0.5]{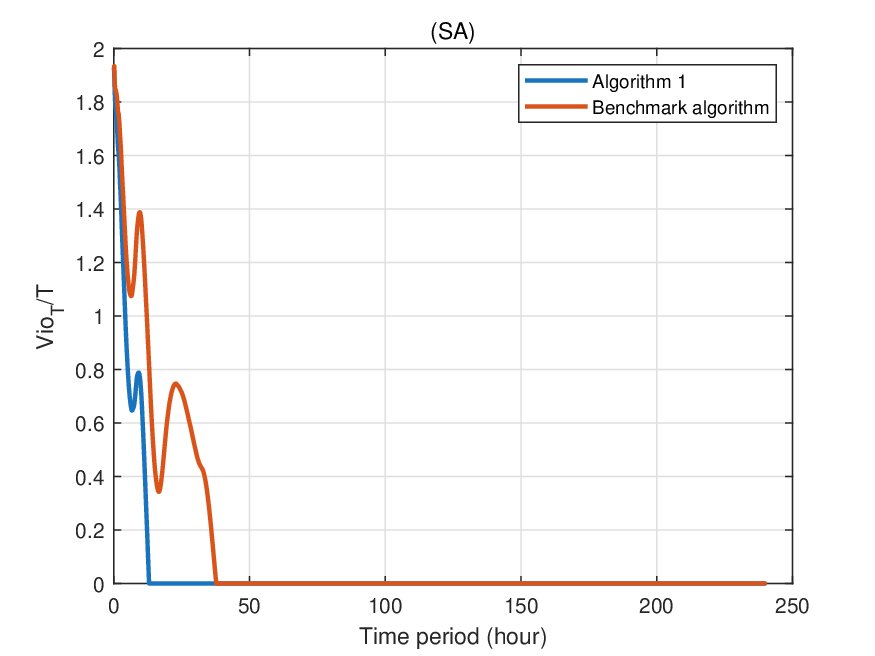}
		%	\caption{XX}
	}
	\caption{Experiment results from SA: (a) dynamic regret, (b) constraint violation. } 
	\label{figSA}
\end{figure}

\begin{figure}[htbp]
	\centering	
	\subfigure[] %第一张子图
	{
		\centering          %子图居中
		\includegraphics[scale=0.5]{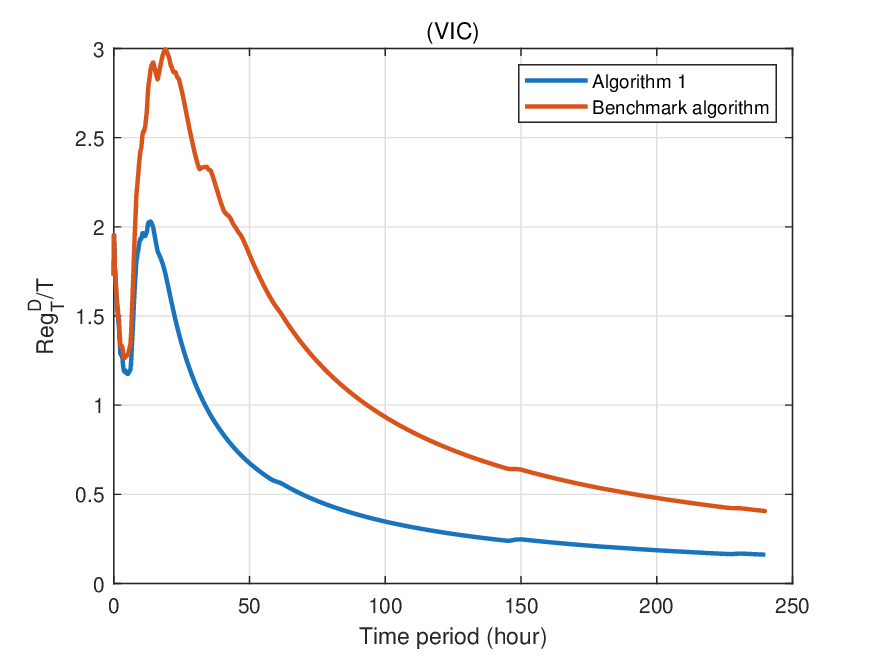}
		%	\caption{XX}
	}
    \subfigure[] %第一张子图
	{
		\centering          %子图居中
		\includegraphics[scale=0.5]{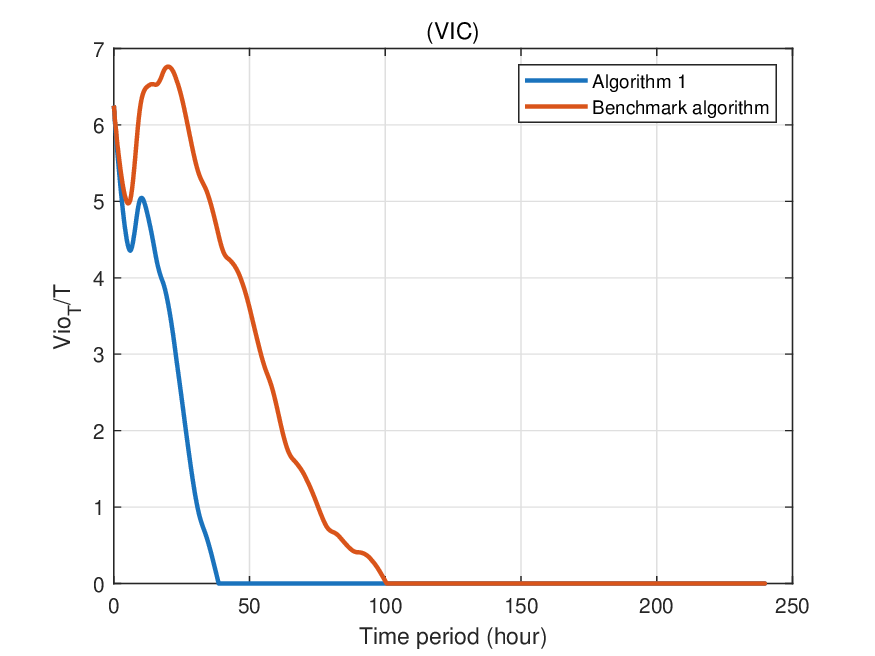}
		%	\caption{XX}
	}
	\caption{Experiment results from VIC: (a) dynamic regret, (c) constraint violation. } 
	\label{figVIC}
\end{figure}

\begin{figure}[htbp]
	\centering	
	\subfigure[] %第一张子图
	{
		\centering          %子图居中
		\includegraphics[scale=0.5]{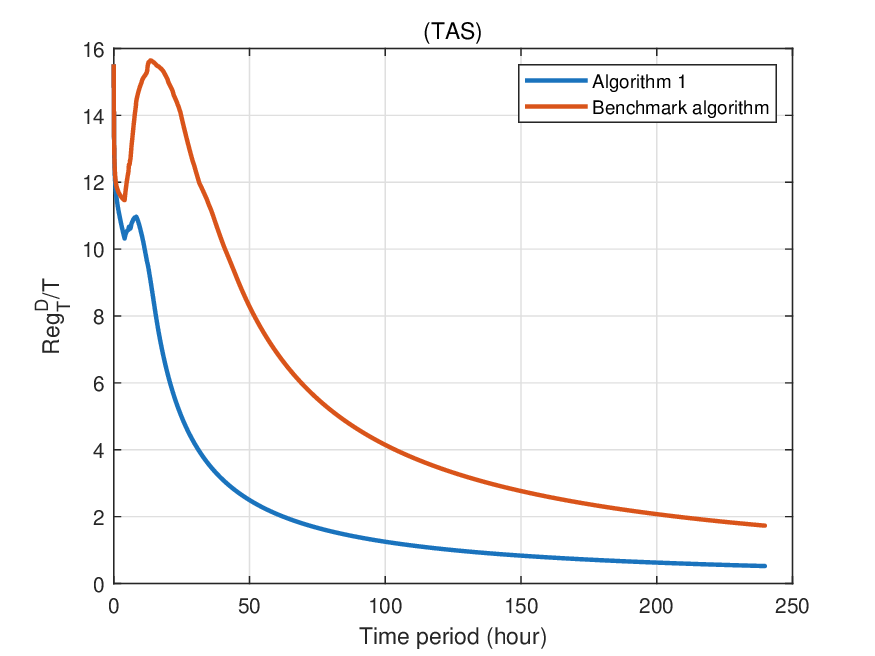}
		%	\caption{XX}
	}
    \subfigure[] %第一张子图
	{
		\centering          %子图居中
		\includegraphics[scale=0.5]{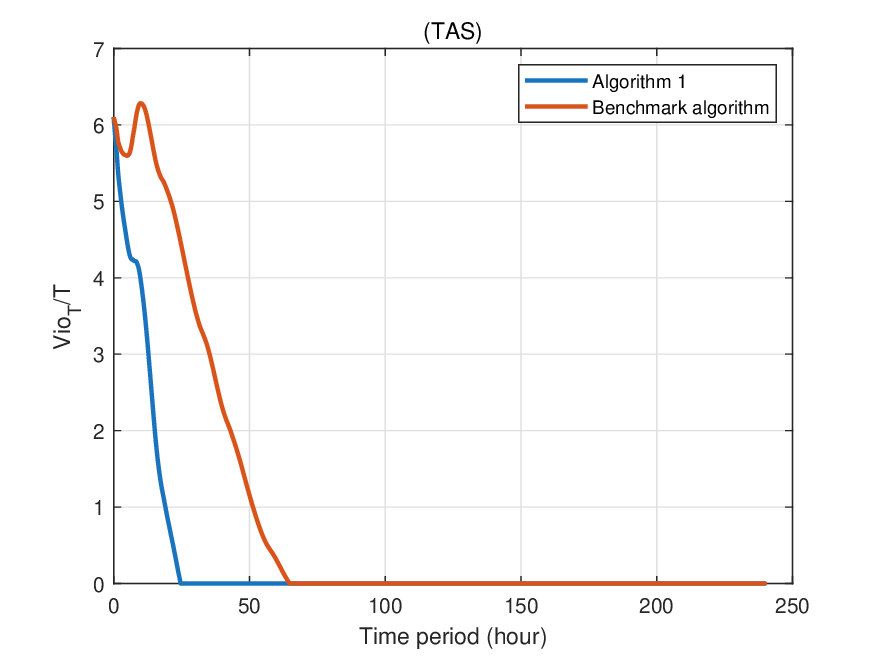}
		%	\caption{XX}
	}
	\caption{Experiment results from TAS: (a) dynamic regret, (c) constraint violation. } 
	\label{figTAS}
\end{figure}

	%%%%%%%%%%%%%%%%%%%%%%%%%%%%%%%%%%%%%%%%%%%%%%%%%%%%%%%%%%%%%%%%%%%%%%%%%
	\section{Conclusions}
	
	In this paper, we study the distributed online economic dispatch problem with time-varying coupled inequality constraints. By integrating a primal-dual framework with a constraint-tracking mechanism, we developed an efficient distributed online algorithm that enables each agent to accurately track the global constraint using only local information exchange. Under suitable assumptions, we established sublinear bounds on both the dynamic regret and the cumulative constraint violation, providing theoretical performance guarantees. Extensive numerical experiments on synthetic and real-world datasets further validated the theoretical findings and demonstrated that the proposed algorithm outperforms existing benchmark methods in terms of convergence behavior and practical efficiency. Several directions remain for future research, including extensions to more general communication graphs and the investigation of distributed online economic dispatch problem with coupled inequality constraints in stochastic and uncertain environments.
	
\section*{Acknowledgments}
	Tao Li was supported by the National Natural Science Foundation of China under Grant No. 62261136550 and the NYU-ECNU Institute of Mathematical Sciences, NYU Shanghai, Shanghai, China. Xiaoqian Wang was supported by the Presidential Foundation of the Academy of Mathematics and Systems Science, Chinese Academy of Sciences, China (No. E555930101).

\bibliographystyle{apalike}
\bibliography{OnlineEconomicDispatch}

\end{document}